\documentclass[12pt]{amsart}

\usepackage{amsthm,amsfonts,amsmath,amsxtra,amscd,amssymb,txfonts,pxfonts,empheq}
\usepackage{hyperref}
\usepackage{xfrac}
\usepackage{graphicx}
\usepackage{url}
\usepackage{mathtools}
\usepackage{verbatim}
\usepackage{mathrsfs}
\usepackage{refcount}
\usepackage{tikz}
\usetikzlibrary{matrix}
\usepackage{tikz-cd}

\usepackage{nicefrac}  % Also try xfrac!!
\usepackage{color}
\newtheorem{proposition}{Proposition}[section]
\newtheorem{theorem}[proposition]{Theorem}
\newtheorem{lemma}[proposition]{Lemma}
\newtheorem{corollary}[proposition]{Corollary}

\newtheorem{definition}[proposition]{Definition}
\newtheorem{formulation}{Problem}
\newtheorem{Notation}{Notation}

\newtheorem{imptheorem}{Theorem}

\newtheorem{impsubtheorem}{Theorem}[imptheorem]

\newtheorem{impsubsubtheorem}{Theorem}[imptheorem]

\newcommand{\repa}{\mathrm{Re}}

\newcommand{\lb}{\left\{}
\newcommand{\rb}{\right\}}

\newcommand{\ls}{\left[}
\newcommand{\rs}{\right]}

\newcommand{\lp}{\left(}
\newcommand{\rp}{\right)}

\newcommand\chars[2]{\left[\begin{smallmatrix}#1\\ #2\end{smallmatrix}\right]}

\newcommand{\Thefunccar}[4]{\Theta\chars{#3}{#4} ({#1};{#2})}
\newcommand{\Thefunc}[2]{\Theta({#1};{#2})}
\newcommand{\Theconst}[3]{\Theta \chars{#2}{#3} ({#1})}
\newcommand{\Tfunc}{T(\xbf)}
\newcommand{\Tfuncnow}[1]{T({#1})}
\newcommand{\Thf}{\Theta_{\permat}}
\newcommand{\Thftilde}{\Theta_{\tilde{\permat}}}

\newcommand{\riemdiv}{\boldsymbol{\Delta}}

\newcommand{\thetadivreal}{(\Thf)_{real}}

\newcommand{\permat}{\boldsymbol{\tau}}
\newcommand{\permatent}{\tau}
\newcommand{\siegel}{\mathcal{H}_g}
\newcommand{\siegelg}{\mathcal{H}_g}

\newcommand{\mfrak}{\mathfrak{M}}
\newcommand{\afrak}{\mathfrak{A}}

\newcommand{\Bcal}{\mathcal{B}}

\newcommand{\homgroup}{H_1(\Gamma, \Zbb)}

\newcommand{\abelmap}{\mathcal{A}}

\newcommand{\Jac}{\mathrm{Jac}}

\newcommand{\Pic}{\mathrm{J_0}}

\newcommand{\srgroup}{\mathbb{G}_g^{\Rbb}}

\newcommand{\srsiegel}{\mathcal{H}^{\Rbb}_{g}}

\newcommand{\matm}{M_{g,\lambda,\varepsilon}}
\newcommand{\matmnow}[3]{M_{{#1},{#2},{#3}}}

\newcommand{\realgroup}{\mathbb{G}_{g,\lambda, \varepsilon}}
\newcommand{\realgroupnow}[3]{\mathbb{G}_{{#1},{#2},{#3}}}

\newcommand{\realsiegel}{\mathcal{H}_{g,\lambda,\varepsilon}}
\newcommand{\realsiegelnow}[3]{\mathcal{H}_{{#1},{#2},{#3}}}

\newcommand{\seto}{\mathcal{O}_{g, \lambda, \varepsilon}}
\newcommand{\setonow}[3]{\mathcal{O}_{{#1},{#2},{#3}}}
\newcommand{\Ocal}{\mathcal{O}}

\newcommand{\sete}{\mathcal{E}_{g, \lambda, \varepsilon}}
\newcommand{\setenow}[3]{\mathcal{E}_{{#1},{#2},{#3}}}
\newcommand{\Ecal}{\mathcal{E}}

\newcommand{\setbnow}[3]{\mathcal{B}_{{#1},{#2},{#3}}}

\newcommand{\settnow}[3]{\mathcal{T}_{{#1},{#2},{#3}}}

\newcommand{\syms}[3]{\mathcal{S}_{{#1},{#2},{#3}}}

\newcommand{\prl}{\boldsymbol{\pi}_{\lambda}}
\newcommand{\prt}{\boldsymbol{\pi}_{2}}
\newcommand{\prlmt}{\boldsymbol{\pi}_{\lambda-2}}

\newcommand{\npos}[1]{p({#1})}
\newcommand{\nneg}[1]{n({#1})}

\DeclareMathOperator{\Rsp}{R}
\DeclareMathOperator{\Isp}{I}

\DeclareMathOperator{\Rcmp}{\mathbf{R}}
\DeclareMathOperator{\Icmp}{\mathbf{I}}

\newcommand{\Zcal}{\mathcal{Z}}

\newcommand{\La}[1]{\Lambda({#1})}
\newcommand{\E}[1]{E({#1})}
\newcommand{\A}[1]{A({#1})}

\newcommand{\Zbb}{\mathbb{Z}}
\newcommand{\Rbb}{\mathbb{R}}
\newcommand{\Cbb}{\mathbb{C}}

\newcommand{\Mat}[2]{\mathrm{Mat}({#1},{#2})}
\newcommand{\GL}[2]{\mathrm{GL}({#1},{#2})}
\newcommand{\Sp}[2]{\mathrm{Sp}({#1},{#2})}

\newcommand{\tp}[1]{{#1}^{\intercal}}
\newcommand{\diag}{\mathrm{diag}}
\DeclareMathOperator{\sgn}{sgn}

\newcommand{\bsalpha}{\boldsymbol{\alpha}}
\newcommand{\bsbeta}{\boldsymbol{\beta}}

\newcommand{\bslambda}{\boldsymbol{\lambda}}
\newcommand{\bsmu}{\boldsymbol{\mu}}

\newcommand{\ebf}{\mathbf{e}}
\newcommand{\mbf}{\mathbf{m}}
\newcommand{\pbf}{\mathbf{p}}
\newcommand{\qbf}{\mathbf{q}}
\newcommand{\rbf}{\mathbf{r}}
\newcommand{\sbf}{\mathbf{s}}

\newcommand{\xbf}{\mathbf{x}}

\newcommand{\zbf}{\mathbf{z}}

\title{Classification of real Riemann surfaces and their Jacobians in the critical case}
\author{Pietro Giavedoni}
\address{E-mail: pietro.giavedoni@gmail.com}

\usepackage[foot]{amsaddr}

\makeindex
\begin{document}

\def\biblio{}

\begin{abstract}
 For every $g\geq 2$ we distinguish real period matrices of real Riemann surfaces of topological type $(g,0,0)$ from the ones of topological type $(g,k,1)$, with $k$ equal to one or two for $g$ even or odd respectively (Theorem \ref{theorem_2}). To that purpose, we exhibit new invariants of real principally polarized abelian varieties of orthosymmetric type (Theorem \ref{theorem_1}). As a direct application, we obtain an exhaustive criterion to decide about the existence of real points on a real Riemann surface, requiring only a real period matrix of its and the evaluation of the sign of at most one (real) theta constant (Theorem \ref{thm_criterion}). A part of our real, algebro-geometric instruments first appeared in the framework of nonlinear integrable partial differential equations.
\end{abstract}

\maketitle

\section{Introduction}

  \label{intro}

\subsection{Formulation of the problem in terms of real Riemann surfaces}

  \label{form_of_prob}
  Let $\Gamma$ be a Riemann surface{\footnote{That is a compact, one dimensional complex manifold.}} of genus $g\geq 1$. Recall that a basis of homological cycles
  \begin{subequations}\label{can_bas}
  \begin{align}
   \Bcal = \lb a_1,a_2, \ldots,a_g; b_1,b_2, \ldots,b_g \rb
  \end{align}
  in $\homgroup$ is said \emph{canonical} if it satisfies the intersection relations 
  \begin{align} \label{int_rel}
   a_j\circ a_k = b_j\circ b_k = 0, \quad\quad a_j\circ b_k = \delta_{jk} \quad\quad\quad j,k=1,2,\ldots g.
  \end{align}
  \end{subequations}
  The basis of \emph{normalized holomorphic differential}
  \begin{align}
    \omega_1, \omega_2, \ldots \omega_g 
  \end{align}
  associated to $\Bcal$ is individuated by the conditions 
  \begin{align}
   \oint_{a_j}\omega_k = \delta_{jk} \quad\quad\quad j,k = 1,2,\ldots,g.
  \end{align}
  Correspondingly, the \emph{period matrix} of $\Gamma$ computed with respect to the basis $\Bcal$ is the $g\times g$-dimensional complex matrix $\permat$ defined elementwise as follows
  \begin{align}
   \permat_{jk}:= \oint_{b_j}\omega_k \quad\quad\quad j,k=1,2,\ldots g.
  \end{align}
  According to a classical result of Riemann, $\permat$ belongs to the \emph{Siegel upper-half space}
  \begin{align}\label{intro:5}
   \siegelg = \lb \permat \in \Mat{g\times g}{\Cbb} \mid \tp{\permat}=\permat, \mathrm{Im} (\permat) \mbox{ positive definite} \rb,
  \end{align}
  of dimension $g$. Let us introduce the modular group
  \begin{align} \label{intro:7}
   \Sp{2g}{\Zbb} = \lb G\in \Mat{2g\times 2g}{\Zbb} \mid G J \tp{G} = J \rb .
  \end{align}  
  where $J$ is the standard, $2g$-dimensional symplectic matrix
  \begin{align} \label{intro:7einhalb}
   J = \ls\begin{smallmatrix} \mathbf{0} & \mathrm{Id}_g \\ -\mathrm{Id}_g & \mathbf{0} \end{smallmatrix}\rs.
  \end{align}
  Denote with $\mfrak$ the modular action 
  \begin{subequations} \label{intro:89} 
   \begin{align} \label{intro:8}
    \mfrak : \Sp{2g}{\Zbb} \times \siegelg \longrightarrow \siegelg
   \end{align}
  of this last one on the Siegel upper-half space defined as follows
  \begin{align} \label{intro:9}
   \mfrak (\ls\begin{smallmatrix} A & B \\ C & D \end{smallmatrix}\rs, \permat) = (A \permat + B)(C \permat + D)^{-1} 
   \quad\quad \quad 
   \permat \in \siegelg, \ls\begin{smallmatrix} A & B \\ C & D \end{smallmatrix}\rs \in \Sp{2g}{\Zbb}.
  \end{align}
 \end{subequations}  
 Another basis $\tilde{\Bcal}$ in $\homgroup$ is also canonical if and only if there exists
 \begin{subequations} \label{trasf_basi}
 \begin{align} \label{trasf_basi_a}
  G = \ls\begin{smallmatrix} A & B \\ C & D \end{smallmatrix}\rs \in \Sp{2g}{\Zbb}
 \end{align}
 such that 
 \begin{align} \label{trasf_Bcal}
  \tilde{\Bcal} = \ls\begin{smallmatrix} D & C \\ B & A \end{smallmatrix}\rs \Bcal.
 \end{align}
 \end{subequations}
 If $\tilde{\permat}$ denotes the corresponding period matrix, one has
 \begin{align} \label{trasf_tau_first}
  \tilde{\permat} = \mfrak (G, \permat).
 \end{align}
 The "viceversa" also holds
 \begin{theorem}[Torelli's Theorem] \label{thm_Torelli}
  Let $\permat, \tilde{\permat}\in \siegelg$ be period matrices of Riemann surfaces $\Gamma, \tilde{\Gamma}$ of genus $g\geq 1$ respectively. These last ones are isomorphic if and only if $\permat$ and $\tilde{\permat}$ belong to the same orbit of the modular action $\mfrak$ of $\Sp{2g}{\Zbb}$ on $\siegelg$.
 \end{theorem}
 According to the previous result, the isomorphism class of a Riemann surface, and so the complete information about its geometry, is contained in any of its period matrices. \\
 
 \vspace{1em}

 Let us now consider the additional datum of an anti-holomorphic involution
 \begin{align}
  \sigma: \Gamma \longrightarrow \Gamma \quad\quad\quad \sigma^2 = \mathrm{Id}_{\Gamma}, \quad \partial \sigma = 0
 \end{align}
 on $\Gamma$. (See \cite{bujalance2010symmetries} for a detailed introduction to the topic.) The couple $(\Gamma, \sigma)$ is referred to as a \emph{real Riemann surface}.  Given another real Riemann surface $(\tilde{\Gamma}, \tilde{\sigma})$, a continuous or holomorphic morphism $\phi:\Gamma\rightarrow \tilde{\Gamma}$ is \emph{real}\footnote{Without an opposite indication, morphisms between real Riemann surfaces are implicitly understood to be real.} if it makes the following diagram commutative
 \begin{center}
    \begin{tikzcd}
        \Gamma \arrow[r, "\phi"] \arrow[d, "\sigma"]
        & \tilde{\Gamma} \arrow[d, "\tilde{\sigma}" ] \\
        \Gamma \arrow[r, "\phi" ]
        & \tilde{\Gamma}
    \end{tikzcd}
 \end{center}
 Points of $\Gamma$ invariant for the action of $\sigma$ are also said \emph{real}. The locus of these last ones consists of finitely many connected components called \emph{ovals}, each of them being homeomorphic to $S^1$. Denote with $k$ their number. Define the quantity $\delta$ to be one if the union of the ovals disconnects $\Gamma$ and zero otherwise. In the former case $(\Gamma, \sigma)$ is said to be of \emph{separated} type. The triple $(g,k,\delta)$ determines the class of real homeomorphism of the real Riemann surface. As such, it is referred to as its \emph{topological type}. Harnack and Weichold proved that a triple $(g,k,\delta)$ is the topological type of some real Riemann surface if and only if it satisfies the following conditions 
 \begin{align}
   (\delta=1, \,\, 1\leq k \leq g+1, \,\, k\equiv g+1 \mod 2) \quad\quad \mbox{or} \quad\quad (\delta=0,\,\, 0\leq k\leq g).
 \end{align}
 
 In particular, a real Riemann surface has real points if and only if its topological type differs from $(g,0,0)$.
 \\
 \vspace{1em}

 The definition of period matrix provided above is insufficient to describe the geometry of a real Riemann surface. Indeed, it is not difficult to find examples of Riemann surfaces bearing several, non equivalent real structures.  This motivates the introduction of some additional requirements on the bases of homological cycles considered. Recall that the anti-holomorphic involution $\sigma$ on $\Gamma$ induces an involution $\sigma_{\star}$ in $\homgroup$.
\begin{definition}
 Let $(\Gamma, \sigma)$ be a real Riemann surface of genus $g\geq 1$. A canonical basis of cycles $\Bcal$ in $\homgroup$ of the form (\ref{can_bas}) satisfying the additional condition 
 \begin{align}
  \sigma_{\star}a_j = a_j \quad\quad j=1,2,\ldots, g.
 \end{align}
 is said \emph{semi-real}. This condition implies
 \begin{align}
  \sigma_{\star}b_j = -b_j +\sum_{k=1}^g M_{jk}a_k \quad\quad j=1,2,\ldots,g,
 \end{align}
 for some $(g\times g)$-dimensional, symmetric matrix $M$ with integral entries. This last one will be referred to as the \emph{reflection matrix of $(\Gamma,\sigma)$ with respect to $\Bcal$}. 
\end{definition}
The existence of semi-real bases of cycles on every real Riemann surface of genus $g\geq 1$ is a classical result, proved, for example, in  \cite{Com155}. Another basis of cycles $\tilde{\Bcal}$ is also semi-real if and only if it is related to $\Bcal$ via a transformation of the form (\ref{trasf_basi}) yielded by some element $G$ of the \emph{semi-real modular group}
\begin{align}
 \srgroup = \lb \ls\begin{smallmatrix} A & B \\ 0 & (\tp{A})^{-1}\end{smallmatrix}\rs  
  \, | \, A \in \GL{g}{\Zbb}, 
  B\in \Mat{g\times g}{\Zbb} \mbox{ and } A\tp{B}=B\tp{A}
 \rb \leq \Sp{2g}{\Zbb}.
\end{align}
In this case, the respective reflection matrices $M$ and $\tilde{M}$ correspond according to the law
\begin{align} \label{new_action}
 \tilde{M} = A M \tp{A} + 2B \tp{A}, \quad\quad G = \ls\begin{smallmatrix} A & B \\ 0 & (\tp{A})^{-1}\end{smallmatrix}\rs \in \srgroup. 
\end{align}
\begin{definition}
 Let $(\Gamma,\sigma)$ be a real Riemann surface of genus at least one. A period matrix of its is said \emph{semi-real} if computed w.r.t. a semi-real basis of cycles $\Bcal$ in $\homgroup$. In this case, one has
 \begin{align} \label{semi_real_form}
  \permat = \nicefrac{1}{2}M + iT,
 \end{align}
 with $M$ the reflection matrix of $(\Gamma, \sigma)$ with respect to $\Bcal$ and $T$ a $g\times g$-dimensional, symmetric and positive definite real matrix.
\end{definition}
For every integer $g\geq 1$, let us define the locus   
\begin{align}
 \srsiegel = \lb \permat \in \siegelg\, | \, 2\repa(\permat)  
 \in \Mat{g\times g}{\Zbb}  
 %\text{ is integer}
 \rb
\end{align}
of all semi-real Riemann matrices of the form (\ref{semi_real_form}) as the \emph{semi-real Siegel upper-half space of dimension $g$.} The modular action $\mfrak$ defined in (\ref{intro:89}) is easily verified to reduce to an action 
\begin{subequations}
\label{azione_semireale}
\begin{align}
 \mfrak: \srgroup\times\srsiegel \longrightarrow \srsiegel 
\end{align}
 of the semi-real modular group on the semi-real Siegel upper-half space. The explicit expression (\ref{intro:9}) simplifies in this case to 
\begin{align} \label{azione_semireale_b}
 \begin{split}
  \mfrak(G, \permat) 
  &= A \permat \tp{A} + B \tp{A}\\
  &= (A\mathrm{Re}(\permat)\tp{A}+B\tp{A})+iA\mathrm{Im}(\permat)\tp{A},
 \end{split}
 \begin{split}
  \permat \in \srsiegel, 
  G=\ls\begin{smallmatrix} A & B \\ 0 & (\tp{A})^{-1}\end{smallmatrix}\rs \in \srgroup
 \end{split}
\end{align}

\end{subequations}
In particular, the doubled real part of $\permat$ undergoes transformation (\ref{new_action}).The following, semi-real version of Theorem \ref{thm_Torelli} holds (see \cite{seppala1989moduli}):
\begin{theorem}[Real Torelli's Theorem] \label{thm_real_torelli}
 Let $\permat, \tilde{\permat}$ be semi-real period matrices of real Riemann surfaces $(\Gamma, \sigma)$, $(\tilde{\Gamma}, \tilde{\sigma})$ of genus $g\geq 2$ respectively. These last ones are real isomorphic if and only if  $\permat$ and $\tilde{\permat}$ lie in the same orbit of the modular action of the group $\srgroup$ on the semi-real Siegel upper-half space $\srsiegel$. 
\end{theorem}
The previous theorem guarantees that complete information about the geometry of a real Riemann surface is contained in any of its semi-real period matrices. Let us recall the relation between the (doubled) real part of the latter - i.e. reflection matrices - and the topological type of the former.
The \emph{real type} of an integral, symmetric matrix $M$  is the triple of integers $(g, \lambda, \varepsilon)$, where $g$ denotes its dimension and the remaining quantities are defined as follows:
\begin{align}\label{def_lambda_eps}
 \lambda(M) = \mathrm{rank} (M\mod 2), 
 \quad \quad \quad 
 \varepsilon(M) = \prod_{k=1}^g (1-M_{kk}) \mod 2. 
\end{align}
The conditions
\begin{align}\label{constraints_rt}
 0 \leq \lambda\leq g 
 \quad\quad 
 \mbox{and}
 \quad\quad
 \begin{cases}
  \lambda \text{ even} & \quad \text{if } \varepsilon=1 \\
  \lambda\geq 1        & \quad \text{if } \varepsilon=0 .
 \end{cases}
\end{align}
determine the eligible values for the real type, this last one being said \emph{orthosymmetric} or \emph{diasymmetric} for $\varepsilon$ equal to one or zero respectively. The real type is easily verified to be invariant for transformations of $M$ of the form (\ref{new_action}) yielded by elements of $\realgroup$. In particular, if $M$ is the reflection matrix of a real Riemann surface, this quantity is independent of the underlying semi-real basis of homological cycles. The real type of a matrix of $\srsiegel$ is understood to be defined as the one of its doubled real part. In view of (\ref{azione_semireale_b}), this last one is also invariant for the modular action $\mfrak$ of $\srgroup$.
\begin{Notation}
 Throughout this article, the letter $g$ is understood to indicate an integer of the form $2g_0$ or $2g_0+1$, for some other integer $g_0$. The same holds for the letter $\lambda$. 
\end{Notation}
\begin{theorem}[See \cite{seppala1989moduli}] \label{theorem_all_poss}
 Let $(\Gamma, \sigma)$ be a real Riemann surface of genus $g\geq 2$. The real type $(g, \lambda, \varepsilon)$ of its reflection matrices - or, equivalently, of its semi-real period matrices - is determined by its topological type $(g, k, \delta)$. In particular, one has 
 \begin{align} \label{tttt}
  (\lambda, \varepsilon) = 
  \begin{cases}
   (g+1-k, \delta) & \quad \text{if } k>0 \\
   (2g_0, 1)       & \quad \text{if } k=0
  \end{cases}
 \end{align} 
\end{theorem}
According to the previous theorem, the topological type of a real Riemann surface of genus $g\geq 2$ can be deduced from a reflection matrix of its, provided that the real type of this last one differs from $(g, 2g_0, 1)$. In the opposite case, the topological type is bound to be either $(g, g+1-2g_0,1)$ or $(g,0,0)$. On the other side,
\begin{theorem} \label{theorem_impossible}
 Let $M$ be a symmetric, integral matrix of real type $(g,2g_0,1)$ for some integer $g\geq 2$. Every real Riemann surface of topological type $(g,g+1-2g_0,1)$ as well as every real Riemann surface of topological type $(g,0,0)$ admits $M$ as reflection matrix w.r.t. some semi-real basis of cycles in the homology. 
\end{theorem}
In other words, if the real type of a reflection matrix is $(g,2g_0,1)$, the reflection matrix alone is insufficient to distinguish between real Riemann surfaces of topological type $(g,0,0)$ and of topological type $(g,g+1-2g_0,1)$. The challenge of this distinction by means of a full semi-real period matrix - possible in view of theorem \ref{thm_real_torelli} - is referred to as the problem of \emph{classification of real Riemann surfaces in the critical case} and represents the main focus of this article. For our investigations it will be convenient to formulate it in terms of some more specific semi-real period matrices. Theorem \ref{theorem_impossible} is an easy consequence of a result of Albert: symmetric, integral matrices of the same real type are equivalent per transformations of the form (\ref{new_action}). This fact allows the introduction of a standard form $\matm$ for each eligible real type satisfying conditions (\ref{constraints_rt}). Our choice in this article will be
 \begin{subequations}
  \label{intro:14}
  \begin{align} \label{intro:14a}
   \matm = 
   \ls\begin{smallmatrix}
   \matmnow{\lambda}{\lambda}{\epsilon} & \mathrm{0} \\ \mathrm{0} & \mathrm{0}
  \end{smallmatrix}\rs
 \end{align}
 where
 \begin{align}\label{intro:14b}
  \matmnow{\lambda}{\lambda}{\epsilon} =
  \begin{cases}
   \ls\begin{smallmatrix} 
   0 & \mathrm{Id}_{\lambda_0} \\
   \mathrm{Id}_{\lambda_0} & 0
   \end{smallmatrix}\rs 
   & \mbox{if  } \varepsilon=1 \\
   \mathrm{Id}_{\lambda} & \mbox{if } \varepsilon = 0.
  \end{cases}
  \end{align}
 \end{subequations} 
\begin{definition}
 Let $(\Gamma,\sigma)$ be a real Riemann surface of topological type $(g,k,\delta)$, with $g\geq 1$. A semi-real basis of cycles in $\homgroup$ is said \emph{real} if its reflection matrix has the standard form $\matm$, $\lambda$ and $\varepsilon$ being determined by the topological type of $(\Gamma,\sigma)$ according to (\ref{tttt}). The corresponding period matrix $\permat$ is also said \emph{real} and takes the form
 \begin{align}
  \permat = \nicefrac{1}{2}\matm + iT
 \end{align}
 for some real, symmetric and positive definite matrix $T$. 
\end{definition}

The introduction of real period matrices suggests to reduce to a \emph{real Siegel upper-half space}
 \begin{align}
  \realsiegel = \lb \permat \in \siegel \mid 2\mathrm{Re} (\permat) = \matm \rb
 \end{align}
 specifically associated to each eligible real type. As a direct consequence of (\ref{tttt}), each of these last ones contains real period matrices of a unique topological type, with the only exception of $\realsiegelnow{g}{2g_0}{1}$, to which both real period matrices of both topological type $(g,2g_0,1)$ and $(g,0,0)$ belong. We are now ready to give a precise formulation of the main problem addressed in this article: 
\begin{formulation}  [Classification of real Riemann surfaces in the critical case]
 \label{formulation_1}
 For every integer $g\geq 2$ distinguish real period matrices of real Riemann surfaces of topological type $(g,0,0)$ from the ones of real Riemann surfaces of topological type $(g,g+1-2g_0,1)$.
\end{formulation}
For each eligible type $(g, \lambda,\varepsilon)$ satisfying conditions (\ref{constraints_rt}) we will also introduce the corresponding \emph{real modular group}
\begin{align}\label{def_real_mod_group}
  \realgroup = \lb 
  %\ls  \begin{array}{cc} A & \nicefrac{1}{2}(M-AM\tp A)(\tp A)^{-1} \\ 0 & (\tp A)^{-1} \end{array} \rs
  \ls\begin{smallmatrix} A & \nicefrac{1}{2}(M-AM\tp A)(\tp A)^{-1} \\ 0 & (\tp A)^{-1} \end{smallmatrix}\rs
  | \,\, 
  A \in \GL{g}{\Zbb}
  ,\,\, 
  M \equiv AM\tp A \,\,\,\mathrm{mod} \,\,\, 2
  \rb \leq \Sp{2g}{\Zbb}.
 \end{align}

Its elements are the matrices of $\srgroup$ preserving the standard form $\matm$ introduced in (\ref{intro:14}) per transformations of the form (\ref{new_action}). As such, via transformations of the form (\ref{trasf_basi}), they also preserve real bases of cycles in the homology of real Riemann surfaces, whose topological type corresponds to $(g,\lambda,\varepsilon)$ according to (\ref{tttt}). Moreover, in view of (\ref{azione_semireale}), the modular action $\mfrak$ defined in (\ref{intro:89}) reduces to an action
\begin{align}
 \mfrak: \realgroupnow{g}{\lambda}{\varepsilon} \times \realsiegelnow{g}{\lambda}{\varepsilon} \longrightarrow \realsiegelnow{g}{\lambda}{\varepsilon}
\end{align}
of the real modular group on the real Siegel upper-half space of the same type.

 \subsection{Formulation of the problem in terms of real Jacobians} 
 \label{subsect_real_jacobians}
 We start by recalling some well-known facts about classical, complex principally polarized abelian varieties (abbr. ppavs) and Jacobians of Riemann surfaces.\\
 Fix an integer $g\geq 1$. A period matrix for a $g$-dimensional ppav $(X,E)$ can be computed following the same procedure described for Riemann surfaces in the previous section, provided that (\ref{int_rel}) is replaced with 
 \begin{align}
  E(a_j,a_k) = E(b_j,b_k) =0, \quad E(a_j,b_k)=\delta_{jk}, \quad\quad j,k=1,2,\ldots g. 
 \end{align}
 To every matrix $\permat$ of the Siegel upper-half space $\siegel$ defined in (\ref{intro:5}) there corresponds a principally polarized abelian variety given by
 \begin{align} \label{tau2ppav_a}
  A(\permat) = \lp {\Cbb^g}/{\La{\permat}}, \E{\tau} \rp
 \end{align}
 Here $\Lambda(\permat)$ is the lattice of $\Cbb^g$ generated by the columns of the matrix $[\mathrm{Id}|\permat]$. These last ones yield a basis in $H_1({\Cbb^g}/{\La{\permat}},\Zbb)$. The polarization $E(\permat)$ is represented by the standard symplectic matrix $J$ defined in (\ref{intro:7einhalb}) with respect to it. The corresponding period matrix coincides with $\permat$. Viceversa, every $g$-dimensional ppav can be represented in the form (\ref{tau2ppav_a}) provided that $\permat$ is a period matrix of its. This last one belongs to $\siegel$. Period matrices of the same ppav are related by a modular transformation of the form (\ref{trasf_tau_first}) and two ppavs are isomorphic if and only if they have the same period matrices. The quotient
  \begin{align}
  \mathcal{A}_g = {\siegelg}/{\Sp{2g}{\Zbb}}
 \end{align}
 is the (rough) moduli space of ppavs of dimension $g$. \\
 The \emph{Jacobian} of a genus $g$ Riemann surface $\Gamma$, denoted with $\Jac(\Gamma)$, is the ppav given in the form (\ref{tau2ppav_a}) with $\permat$ a period matrix of $\Gamma$. Consequently, a Riemann surface and its Jacobian have the same period matrices. An alternative realization of $\Jac(\Gamma)$ is provided by the quotient 
 \begin{align} \label{def_Pic}
  J_0(\Gamma) = \mathrm{Div}_0(\Gamma)/\mathrm{Princ}(\Gamma)
 \end{align}
 of divisors of zero degree on $\Gamma$ with respect to the principal ones. The Abel map 
 \begin{subequations}
  \label{def_Abel_map}
  \begin{align}\label{def_Abel_map_a}
   \mathcal{A}: \mathrm{Div}_0(\Gamma) \longrightarrow \Cbb^g/\Lambda (\permat)
  \end{align}
  defined as follows
  \begin{align}\label{def_Abel_map_b}
   \mathcal{A}(D) = \ls\sum_{n=1}^N \int_{Q_n}^{P_n}\boldsymbol{\omega}\rs_{\Lambda(\permat)}, \quad\quad\quad D = \sum_{n=1}^NP_n - \sum_{n=1}^N Q_n \in \mathrm{Div}_0(\Gamma)
  \end{align}
  induces an isomorphism of groups 
  \begin{align}
   \mathcal{A}: J_0(\Gamma) \longrightarrow \Cbb^g/\Lambda (\permat).  
  \end{align}
 \end{subequations}
 The Torelli mapping 
 \begin{align}
  \mathcal{T}: \mathcal{M}_g \longrightarrow \mathcal{A}_g 
 \end{align}
 defined on the moduli space of Riemann surfaces and associating to each of these last ones its Jacobian is known, by the homonymous theorem \ref{thm_Torelli}, to be injective. The question to describe its image is the classical Schottky problem (see \cite{ky2012schottky} and references therein).

 Let us now consider the real framework (\cite{silhol1992compactifications}, \cite{Com155}, \cite{comessatti1926sulle}): a \emph{real ppav} is a pair $(A,S)$, where $A=(X,E)$ is a complex ppav, $S$ is an anti-holomorphic involution on $X$ such that $S(0)=0$ and the following compatibility condition with the polarization $E$ holds 
 \begin{align}
  E(S_{\star}(c_1), S_{\star}(c_2)) = -E(c_1,c_2), \quad\quad\quad c_1,c_2 \in H_1(X,\Zbb).
 \end{align} 
 A period matrix of a real ppav is said \emph{semi-real} if the underlying basis $\{a_1,a_2,\ldots a_g,b_1,b_2,\ldots,b_g\}$ in $H_1(X,\Zbb)$ satisfies the conditions
 \begin{align}
  S_{\star}(a_j) = a_j \quad\quad j=1,2,\ldots,g.
 \end{align}
 To every matrix $\permat$ of the semi-real Siegel upper-half space $\srsiegel$ there corresponds the real ppav
 \begin{subequations} \label{concrete_form}
  \begin{align} \label{concrete_form_a}
   (\A{\permat}, S_{conj})
  \end{align}
  where $\A{\permat}$ is the (complex) ppav associated to $\permat$ via (\ref{tau2ppav_a}) and $S_{conj}$ the anti-holomorphic involution on $\Cbb^g/\Lambda (\permat)$ induced by the standard complex conjugation on $\Cbb^g$:
  \begin{align}\label{concrete_form_b}
   S_{conj}: 
   [\zbf]_{\Lambda(\permat)} \longrightarrow [\overline{\zbf}]_{\Lambda(\permat)} 
   \quad\quad\quad [\zbf]_{\Lambda (\permat)} \in \Cbb^g/\Lambda(\permat).
  \end{align}
 \end{subequations}
 $\permat$ turns out to be a semi-real period matrix of (\ref{concrete_form_a}).
 Viceversa, every real ppav $(X,E,S)$ can be obtained in the form (\ref{concrete_form}) provided that $\permat$ is a semi-real period matrix of its. This last one belongs to $\srsiegel$. Semi-real period matrices of the same real ppav are related by a modular transformation of the form (\ref{azione_semireale}) and two real ppavs are isomorphic if and only if they have the same semi-real period matrices. The quotient 
 \begin{align}
  \mathcal{A}_{g}^{\Rbb}:=  \srsiegel/\srgroup
 \end{align}
 is the (rough) moduli space of real ppavs of dimension $g$. Summarizing, $\srsiegel$ and $\srgroup$ play in the real case the same role as $\Sp{2g}{\Zbb}$ in the complex one.

 \vspace{1em}

 In view of its invariance with respect to the modular action of $\srgroup$, the real type $(g,\lambda, \varepsilon)$ of a real ppav can be defined as the one of its semi-real period matrices. Moreover, a geometrical interpretation of it is possible: apart from the dimension $g$, the locus of real points of a real ppav has $2^{g-\lambda}$ connected components and $\varepsilon$ expresses some real properties of its polarization. Real ppavs of a given real type $(g,\lambda, \varepsilon)$ form a connected component of $\mathcal{A}_g^{\Rbb}$. Resorting to real period matrices, this last one can be realized as the quotient 
 \begin{align}
  \mathcal{A}_{g,\lambda,\varepsilon} := \realsiegel/\realgroup
 \end{align}

 \vspace{0.5em}

 The Jacobian of a real Riemann surface $(\Gamma,\sigma)$ inherits a real structure $S_{\Gamma}$. Indeed, $\sigma$ extends in a natural way to divisors of degree zero, to the quotient $J_0(\Gamma)$ and via the Abel map (\ref{def_Abel_map}) to $\Jac(\Gamma)$. The couple $(\Jac(\Gamma), S_{\Gamma})$ is a real ppav, to be expressed in the form (\ref{concrete_form}) with $\permat$ a semi-real period matrix of $(\Gamma,\sigma)$. In particular, a real Riemann surface and its real Jacobian have the same semi-real period matrices. \\
 In this framework, theorems \ref{theorem_all_poss} and \ref{theorem_impossible} above can be rephrased as follows: each component of $\mathcal{A}_g^{\Rbb}$ contains real Jacobians of a single topological type, with the only exception of $\mathcal{A}_{g,2g_0,1}$, to which real Jacobians of topological type 
 $(g,g+1-2g_0,1)$ as well as $(g,0,0)$ belong. A distinction based solely on the real type is in this last case not possible. Problem \ref{formulation_1}  formulated in the previous section is equivalent to the following
 \begin{formulation}[Classification of real Jacobians in the critical case]\label{formulation_2}
  For every integer $g\geq 2$ distinguish real period matrices of real Jacobians of topological type $(g,0,0)$ from the ones of real Jacobians of topological type $(g, g+1-2g_0,1)$.
 \end{formulation}
Notice that problem \ref{formulation_1} and \ref{formulation_2} above have the same solution but they are not the "same" problem. Indeed, given a real Jacobian, the challenge to effectively determine the corresponding real Riemann surface is non trivial and to the best of our knowledge still open.

 \subsection{Summary of results and methods}
 The formulation and the first results about the critical case of classification of real Jacobians are due to Comessatti, although the problem might have been known already to Klein, 
 at the end of the nineteenth century \cite{klein1893}. In \cite{Com155} he provides a solution to the case of genus two in terms of the determinant of the 
 imaginary part of real period matrices. An alternative proof of this result based on different instruments is available in \cite{silhol1992compactifications}. \\
 The question is closely related to other research in real algebraic geometry,
 for example the one of \cite{lattarulo2003schottky}, \cite{silhol2001schottky} and \cite{bochnak1997morphisms}. Nevertheless, to the best of our knowledge, 
 no significant progress has been made since the time of Comessatti for genus higher than two. In this article, we recast his results in terms of theta-constants 
 and extend them to all genera. \\
 \vspace{0.25em}
 
 Recall that, for every integer $g\geq 1$, the $g$-dimensional theta function
 \begin{subequations}
  \label{def_theta}
  \begin{align} \label{def_theta_a}
   \Theta\chars{\bsalpha}{\bsbeta}: \Cbb^g\times \siegel \longrightarrow \Cbb
  \end{align}
  with integral\footnote{In general, characteristics can take real values but only integral ones are relevant for our purposes.} characteristics $\bsalpha,\bsbeta\in \Zbb^g$ is defined as follows
  \begin{align} \label{def_theta_b}
   \Thefunccar{\zbf}{\permat}{\bsalpha}{\bsbeta} &= \sum_{\mbf \in \Zbb^g} 
   \exp\ls \pi i \tp{\lp\mbf+\nicefrac{\bsalpha}{2}\rp} \permat \lp\mbf +\nicefrac{\bsalpha}{2}\rp + 2\pi i \tp{\lp\zbf+\nicefrac{\bsbeta}{2}\rp}\lp\mbf + \nicefrac{\bsalpha}{2}\rp \rs, \quad\quad 
   \zbf\in \Cbb^g, \permat \in \siegel.
  \end{align}
  As usual, we will put
  \begin{align} \label{def_theta_c}
   \Thefunc{\zbf}{\permat} = \Thefunccar{\zbf}{\permat}{\mathrm{0}}{\mathrm{0}}, \quad\quad\quad \zbf \in \Cbb^g, \permat \in \siegel.
  \end{align}
 \end{subequations}
The action 
\begin{subequations}
\label{def_afrak}
\begin{align}\label{def_afrak_a}
 \afrak: \Sp{2g}{\Zbb}\times \Zbb^{2g} \longrightarrow \Zbb^{2g}
\end{align}
of the modular group on the set of theta-characteristics is defined as follows:
\begin{align} \label{def_afrak_b}
 \afrak \lp 
 \ls\begin{smallmatrix} A & B \\ C & D \end{smallmatrix} \rs,
 \ls \begin{smallmatrix} \bsalpha \\ \bsbeta \end{smallmatrix}\rs
 \rp
 =
 \ls\begin{smallmatrix} D & -C \\ -B & A \end{smallmatrix} \rs
 \ls\begin{smallmatrix} \bsalpha \\ \bsbeta \end{smallmatrix}\rs
 +
 \ls\begin{smallmatrix}
 \diag(C\tp{D})\\
 \diag(A\tp{B})
 \end{smallmatrix}\rs,
 \quad \quad \quad  
 \ls\begin{smallmatrix}
     \bsalpha \\ \bsbeta
 \end{smallmatrix}\rs \in \Zbb^{2g}, \,\, \,
 \ls\begin{smallmatrix}
 A & B \\ C & D
 \end{smallmatrix}\rs \in \Sp{2g}{\Zbb}.
\end{align}
\end{subequations}
Theta-characteristics are said \emph{even} (resp. \emph{odd}) if the product $\tp{\boldsymbol{\alpha}}\cdot \boldsymbol{\beta}$ is even (resp. odd). This property is preserved by the action $\mathfrak{A}$ of $\Sp{2g}{\Zbb}$. For a transformation of the arguments 
\begin{subequations}
 \label{trasf_theta}
 \begin{align}\label{trasf_theta_a}
  \tilde{\permat} = \mfrak(G,\permat), \quad 
  \tilde{\zbf} = [\tp{(C\permat+D)}]^{-1}\zbf, \quad
  \chars{\tilde{\bsalpha}}{\tilde{\bsbeta}} =
  \afrak\lp G, \chars{\bsalpha}{\bsbeta}\rp,
  \quad\quad\quad
  G = \ls\begin{smallmatrix} A & B \\ C & D \end{smallmatrix} \rs\in \Sp{2g}{\Zbb},
 \end{align}
 the value of the theta function changes as follows 
 \begin{align} \label{trasf_theta_b}
  \Thefunccar{\tilde{\zbf}}{\tilde{\permat}}{\tilde{\bsalpha}}{\tilde{\bsbeta}}
  =
  \kappa
  [\det(C\permat+D)]^{\nicefrac{1}{2}}
  \exp\lb \frac{1}{2}\sum_{i\leq j}z_iz_j\frac{\partial \log\det(C\permat+D)}{\partial \permatent_{ij}} \rb
  \Thefunccar{\zbf}{\permat}{\bsalpha}{\bsbeta}  
 \end{align}
 where $\kappa$ is understood to be a nonvanishing complex number independent of $\permat$ and $\zbf$.
\end{subequations}
In the present article, a fundamental role will be played by theta-constants 
\begin{align}\label{def_theta_const}
 \Theconst{\permat}{\bsalpha}{\bsbeta} = \Thefunccar{\mathbf{0}}{\permat}{\bsalpha}{\bsbeta} \quad\quad\quad \permat \in \realsiegel
\end{align}
with integer characteristics $\bsalpha, \bsbeta \in \Zbb^g$.\\
\vspace{0.5em}
 For any real type $(g,\lambda,\epsilon)$ satisfying constraints (\ref{tttt}), let us introduce the sets
 \begin{subequations}\label{def_set_oe}
  \begin{align} 
   \setcounter{equation}{14} \label{def_set_o}
   \seto &= \lb [\bsbeta]_2 \in \Zbb_2^g \mid \tp{\bsbeta}\matm \bsbeta\equiv 2 \mod 4; \,\,\,  \beta_{\lambda+1}\equiv \beta_{\lambda+2}\equiv \ldots \equiv \beta_g \equiv 0 \mod 2 \rb\\
   \setcounter{equation}{4} \label{def_set_e}
   \sete &= \lb [\bsbeta]_2 \in \Zbb_2^g \mid \tp{\bsbeta}\matm \bsbeta\equiv 0 \mod 4; \,\,\,  \beta_{\lambda+1}\equiv \beta_{\lambda+2}\equiv \ldots \equiv \beta_g \equiv 0 \mod 2 \rb
  \end{align}
 \end{subequations}
 Correspondingly, for every $\permat$ in $\realsiegel$, let us consider the indexed families\footnote{
 Formally, $\Ocal(\permat)$ can be thought of as a function 
 \begin{align}
  \seto \longrightarrow \Rbb
 \end{align}
 associating to every $[\bsbeta]_2$ in $\seto$ the value of the theta-constant $\Theconst{\permat}{\mathrm{0}}{\bsbeta}$. An analogous definition holds for $\Ecal(\permat)$.}
 \begin{align} \label{def_o_tau}
  \Ocal (\permat) = \lb\Thefunccar{\mathrm{0}}{\permat}{\mathrm{0}}{\bsbeta}\rb_{[\bsbeta]_2\in \seto},
  \quad\quad \quad \quad\quad
  \Ecal (\permat) = \lb\Thefunccar{\mathrm{0}}{\permat}{\mathrm{0}}{\bsbeta}\rb_{[\bsbeta]_2\in \sete}.
 \end{align} 
 Notice that no ambiguity is contained in this definition, due to the well-known identity
 \begin{align}\label{form_segno_theta_char}
   \Thefunccar{\zbf}{\permat}{\bsalpha + 2\mathbf{d}}{\bsbeta + 2\mathbf{f}} = \exp(\pi i \tp{\mathbf{f}}\bsalpha) \Thefunccar{\zbf}{\permat}{\bsalpha}{\bsbeta}, 
 \end{align}
 for all $\permat\in \siegel$; $\zbf \in \Cbb^g$; $\bsalpha,\bsbeta \in \Rbb^g$;  $\mathbf{d},\mathbf{f} \in \Zbb^n$. This last one guarantees that the theta-constants 
 are independent of the particular representative $\bsbeta$ of the class $[\bsbeta]_2$. The value of the triple $(g,\lambda, \varepsilon)$ on the r.h.s. of 
 (\ref{def_o_tau}) is understood to be the real type of the real Riemann matrix $\permat$ on the l.h.s..

 \vspace{1em}
 
  The first part of this article contains results about real ppavs, mostly of orthosymmetric type, employed in the solution of the critical case of classification of real Jacobians and also relevant by themselves. Our first result is the following

 \setcounter{imptheorem}{1}
 
 \begin{impsubtheorem}[See Thm. \ref{itc:theorem_100} and Lemma      \ref{rvt:lemma_1}] \label{theorem_1}
  Fix an orthosymmetric real type $(g,\lambda,1)$. Let $\permat$ be a matrix of the real Siegel upper-half space $\realsiegelnow{g}{\lambda}{1}$. The elements of the families $\Ocal(\permat)$ and $\Ecal(\permat)$ are all real. Moreover, let $\tilde{\permat}$ be another matrix of $\realsiegelnow{g}{\lambda}{1}$ such that
  \begin{align}
   \tilde{\permat} = \mfrak(G,\permat) \quad\quad \mbox{for some } \quad G \in \realgroupnow{g}{\lambda}{1}. 
  \end{align}
  Then 
  \begin{align}
   \quad\quad\quad\quad
   \quad
   \Ocal(\tilde{\permat}) = \Ocal(\permat) 
   \quad \mbox{and} \quad
   \Ecal(\tilde{\permat}) = \Ecal(\permat)
   \quad\quad 
   \mbox{up to permutations of the indexes}.
  \end{align}
  Consequently, the elements of $\Ocal(\permat)$ and $\Ecal(\permat)$ respectively constitute multi-dimensional, real quantities associated to the real ppav $(A(\permat), S_{conj})$ defined in (\ref{concrete_form}), invariant - up to permutations - per real isomorphisms. 

 \end{impsubtheorem}

 To the best of our knowledge, these invariants are new in the literature and for the first time here recognised as such. Their relevance goes beyond the classification of 
 real Jacobians. For example, they provide an instrument to distinguish non-isomorphic real ppavs of the same orthosymmetric type. This clarifies our motivation to 
 define $\Ocal(\permat)$ and $\Ecal(\permat)$ not simply as sets but as indexed families: in this way, a trace is kept of possible repetitions among the values of
 their entries, potentially crucial for their optimal employment.  
  
 The first part of the proof of theorem \ref{theorem_1} consists in providing a simpler version of the transformation law (\ref{trasf_theta}) of the theta function for the real framework. We will restrict to characteristics of the form $\tp{[\tp{\mathbf{0}},\tp{\boldsymbol{\beta}}]}$. Notice that the set of these last ones is invariant for the action of $\afrak$. With the obvious identification 
\begin{align}
 \chars{\mathbf{0}}{\bsbeta}\longleftrightarrow \bsbeta \quad\quad\quad \bsbeta\in \Zbb^g
\end{align}
and considering classes modulo 2 of these vectors, $\afrak$ reduces to an action 
\begin{subequations}
\label{def_afrak_rid}
 \begin{align}
  \afrak: \realgroup \times \Zbb_2^g \longrightarrow \Zbb_2^g
 \end{align}
 given explicitely by
 \begin{align}
  \afrak (G, [\bsbeta]_2) = [A\bsbeta - \nicefrac{1}{2}\diag(M-AM\tp{A})]_2,
  \quad\quad
  G = \ls  \begin{smallmatrix} 
  A & \nicefrac{1}{2}(M-AM\tp A)(\tp A)^{-1} \\ 
  0 & (\tp A)^{-1} 
  \end{smallmatrix} \rs.
 \end{align}
\end{subequations}
\begin{impsubsubtheorem}[See Thm. \ref{itc:prop_20}] 
 \label{th_inv_theta_val}
 Let $[\bsbeta]_2$, $[\tilde{\bsbeta}]_2$ be vectors of $\Zbb_2^g$ and $\permat$, $\tilde{\permat}$ be matrices of some real Siegel upper-half plane $\realsiegel$ 
 of dimension $g$. Assume that %, for some element $G$ of the corresponding real modular group $\realgroup$ one has
\begin{align}
 \tilde{\permat} = \mfrak(G, \permat) 
 \quad \mbox{and} \quad
 [\tilde{\bsbeta}]_2 = \afrak (G,[\bsbeta]_2)
 \quad\quad\mbox{for some} \quad
 G \in \realgroup.
\end{align}
Then
\begin{align}
 \Theconst{\tilde{\permat}}{\mathrm{0}}{\tilde{\bsbeta}}
 =
 \Theconst{\permat}{\mathrm{0}}{\bsbeta}.
\end{align}
\end{impsubsubtheorem}

Next, we restrict to orthosymmetric types and prove a property of the action $\afrak$:
%\begin{theorem}[\tc{red}{cfr ...}]
\begin{impsubsubtheorem}[See Thm. \ref{itc:theorem_1}]
 \label{th_inv_theta_char}
 Let $(g,\lambda, 1)$ be an orthosymmetric real type. The set $\setonow{g}{\lambda}{1}$ is invariant for the action 
 $\afrak$ of $\realgroupnow{g}{\lambda}{1}$ on $\Zbb_2^g$ defined in (\ref{def_afrak_rid}). The same is true for $\setenow{g}{\lambda}{1}$. 
%\end{theorem}
\end{impsubsubtheorem}
This result is a generalization of the classical invariance of the set of odd (resp. even) characteristics for the action $\afrak$ of the modular group $\Sp{2g_0}{\Zbb}$ 
defined in (\ref{def_afrak}).  In particular, the elements of 
$\setonow{g}{\lambda}{1}$ (resp. $\setenow{g}{\lambda}{1}$) can be considered as the real, orthosymmetric analogue of the former ones. Theorem \ref{theorem_1} follows easily from a 
combination of theorems \ref{th_inv_theta_val} and \ref{th_inv_theta_char}. As a next step, we investigate the sign of real theta functions. 
\begin{impsubtheorem}[See Theorem \ref{rvt:thm_1}]
 \label{intro:thm_last}
 Let $(g,\lambda,1)$ be an admissible, orthosymmetric real type, with $\lambda$ (even,) equal to $2\lambda_0$ and strictly positive. Introduce the set
 \begin{align} \label{def_T_cal}
  \settnow{g}{\lambda}{1} =
  \lb
   [\bsbeta]_2 \in \setenow{g}{\lambda}{1} \,\, | \,\, \beta_j\cdot \beta_{j+\lambda_0}\equiv 0 \,\,\,\mathrm{mod} \,\,\, 2, j=1,2,\ldots,\lambda_0
  \rb.
 \end{align}
 Let $\permat$ be a matrix of the corresponding Siegel upper-half space $\realsiegelnow{g}{\lambda}{1}$. One has 
 \begin{align} \label{ineq_theta_intro}
  \sum_{[\bsbeta]_2\in \settnow{g}{\lambda}{1}} \Theconst{\tau}{0}{\bsbeta} > 0.
 \end{align}
 In particular, at least one element of the indexed family $\Ecal(\permat)$ is strictly positive in this case.
\end{impsubtheorem}
           
Finally, we consider the zeros of the theta function
\begin{impsubtheorem}[See Theorem \ref{ztf:thm_1}]
 \label{intro:thm_notin}
 Let $\permat$ be a matrix of $\realsiegelnow{g}{\lambda}{1}$, for some admissible, orthosymmetric real type $(g,\lambda,1)$. Denote with $\Rcmp_{\notni\mathbf{0}}$ a connected component of the locus of real points on the real ppav $(A(\permat), S_{conj})$ non containing its zero. One has
 \begin{align}
  \dim ((\Theta_{\permat}) \cap \Rcmp_{\notni\mathbf{0}}) = g-1.
 \end{align}
 Here $(\Theta_{\permat})$ is understood to indicate the theta-divisor on $A(\permat)$.
\end{impsubtheorem}
Notice that $(A(\permat), S_{conj})$ is not required to be a real Jacobian. The prove of the previous theorem relies essentially on an appropriate application of  the elementary properties of the theta function.

\setcounter{imptheorem}{1}
\setcounter{impsubtheorem}{0}

 \vspace{1em}
 In the second part of the article, we apply the material above to solve the critical case of classification of real Jacobians. The indexed family $\Ocal(\permat)$ proves to be the right instrument to that purpose.
\begin{imptheorem}[See Theorem \ref{fin:thm_ultimo} and Theorem  \ref{thm_E_quantities}] 
 \label{theorem_2}
 Let $\permat$ be a real period matrix of a real Riemann surface $(\Gamma,\sigma)$ of genus $g$ equal to $2g_0$ or $2g_0+1$, for some integer $g_0\geq 1$. Assume that $\permat$ belongs to the real Siegel upper-half space $\realsiegelnow{g}{2g_0}{1}$. One has
 \begin{enumerate}
  \item  If the topological type of $(\Gamma, \sigma)$ is $(g,g+1-2g_0,1)$, then all the elements of $\Ocal(\permat)$ are strictly negative.
  \item If the topological type of $(\Gamma, \sigma)$ is $(g,0,0)$, then all elements of $\Ocal(\permat)$ are nonnegative.
 \end{enumerate}

\end{imptheorem}
 Let us consider the real and imaginary subspaces 
 \begin{align} \label{intro:def_RI}
  \Rsp = \{\xbf \mid \xbf \in \Rbb^g\}, \quad\quad\quad\quad \Isp = \{ i \xbf \mid \xbf \in \Rbb^g \}.
 \end{align}
 of $\Cbb^g$. Let us introduce the notation 
 \begin{align} \label{intro:def_Thf}
  \Thf (\zbf) = \Thefunc{\zbf}{\permat} \quad\quad\quad\quad \zbf \in \Cbb^g
 \end{align}
 for theta as a function of the only variable $\zbf$, if $\permat$ is a fixed element of $\siegel$.  
 An important ingredient to prove theorem \ref{theorem_2} is the following

 \begin{impsubtheorem} [See Proposition \ref{ztf:prop_100} and Theorem \ref{ztf:real_case:thm_10}] \label{intro:th_zeri_RI}
  Let $(\Gamma,\sigma)$ be a real Riemann surface of genus $g$ equal to $2g_0$ or $2g_0+1$, for some integer $g_0\geq 1$. Denote with $\permat$ a real period matrix of its.
  \begin{enumerate}
   \item If the topological type of $(\Gamma,\sigma)$ is $(g,g+1-2g_0,1)$, then the function $\Thf$ has no zeros on the imaginary subspace $\Isp$.
   \item If the topological type of $(\Gamma,\sigma)$ is $(g,0,0)$, then the set
   \begin{align}
    \{ \zbf \in \Rsp \mid \Thf (\zbf)\neq 0 \}
   \end{align}
   is path-connected.
  \end{enumerate}
 \end{impsubtheorem}
 This theorem is largely inspired by the theory of algebro-geometric solutions to integrable, nonlinear PDEs (see \cite{belokolos1994algebro}, \cite{dubrovin1981theta} 
for a general introduction). Its first part has already been proved for example in \cite{dubrovin1988real} with some additional assumptions on the underlying real basis of homological cycles. We remove them here. The second part is inspired by the work of Malanyuk \cite{malanyuk1994finite}, \cite{malanyuk1991finite}, whereas its original statement seems to us to be incorrect.\footnote{
According to our understanding, it is claimed there (with different normalization conventions), that the theta function does not vanish, under appropriate hypotheses,
on the set $\Rsp_1$ introduced in section  \ref{ztf}. But this is impossible, because that set contains points corresponding to odd theta characteristics.
} For a modified statement (Point 2 in the previous theorem), we provide  proof based on theorem (\ref{intro:thm_notin}) above and on Riemann theorem about the zeros of the theta function.

\vspace{0.5em}

It is not difficult to verify that $\Thf$ is real on both $\Rsp$  and $\Isp$, if $\permat$ belongs to the real Siegel upper-half space $\realsiegelnow{g}{2g_0}{1}$ (see lemma \ref{rvt:lemma_1}).
From Theorem \ref{intro:th_zeri_RI} one deduces that $\Thf$ cannot change sign on $\Rsp$ or $\Isp$, if $\permat$ is a period matrix of topological type $(g,0,0)$ or $(g,g+1-2g_0,1)$
respectively. Under the same hypothesis, theorem \ref{intro:thm_last} implies that $\Thf$ takes a strictly positive value at some point of $\Rsp$ as well as at some point of $\Isp$. Theorem \ref{theorem_2} is then obtained by means of some appropriate manipulations based on the  properties of the theta function.

\vspace{1em}

When considered in the general framework exposed in section \ref{form_of_prob}, this last result implies the following criterion to decide about the existence of real points on a real Riemann surface.

\begin{imptheorem}\label{thm_criterion}
 Let $(\Gamma, \sigma)$ be a real Riemann surface of genus $g$ equal to $2g_0$ or $2g_0+1$ for some integer $g_0\geq 1$. Let $M$ be the reflection matrix of $(\Gamma,\sigma)$ w.r.t. some semi-real basis of cycles in $\homgroup$.
 \begin{enumerate}
  \item If $\lambda(M)\neq 2g_0$ or $\varepsilon(M)\neq 1$, then $(\Gamma,\sigma)$ has real points.
  \item If $\lambda(M)= 2g_0$ and $\varepsilon(M)=1$, 
  let $\permat$ be a real period matrix of $(\Gamma,\sigma)$ and $\bsbeta$ any vector in $\{0,1\}^g$ satisfying the conditions
  \begin{align}
   \tp{\bsbeta}\matmnow{g}{2g_0}{1}\bsbeta \equiv 2 \mod 4, 
   \quad\quad
   \beta_g=0 \quad \text{if g is odd.} 
  \end{align}
  where the matrix $\matmnow{g}{2g_0}{1}$ is defined in (\ref{intro:14}). In this case $(\Gamma,\sigma)$ has real points if and only if 
  \begin{align}
   \Theconst{\permat}{\mathbf{0}}{\bsbeta}<0.
  \end{align}
 \end{enumerate}
\end{imptheorem}

Notice that our criterion requires only the knowledge of the quantities $\lambda$ and $\epsilon$, 
and of the sign of at most one (real) theta-constant. Its implementation as 
numerical algorithm and an extensive exam of its potential for the investigation of systems of multivariate real polynomial equations are currently part of a work in progress.

\section{Invariant theta constants for real ppavs}

\label{itc}

The main purpose of this section is the introduction of the indexed families $\Ocal (\permat)$ and $\Ecal (\permat)$ as invariants -modulo permutations of the indexes- 
of real ppavs of orthosymmetric type (Theorem \ref{theorem_1}). We start with a result holding also in the diasymmetric case.
\begin{theorem}[See Thm. \ref{th_inv_theta_val}]
 \label{itc:prop_20}
 Let $(g,\lambda,\epsilon)$ be a type of real ppavs, with $g\geq 2$. Denote with $M$ the corresponding matrix $\matm$, defined in (\ref{intro:14}). Let $G$ be an element of the real modular group $\realgroup$. Recall that in this case there exists a matrix 
 \begin{align}\label{itc:prop_20:form_1}
  A \in \GL{g}{\Zbb} \quad\quad \mbox{ satisfying } \quad\quad  AM\tp{A}\equiv M \mod 2 
 \end{align}
 such that  
 \begin{align}\label{itc:prop_20:form_2}
  G = \ls \begin{smallmatrix}  
           A & \nicefrac{1}{2}(M-AM\tp{A})(\tp{A})^{-1} \\
           \mathrm{0} & (\tp{A})^{-1}
          \end{smallmatrix}
 \rs
 \end{align}
 Let $\permat, \tilde{\permat}$ be real Riemann matrices of the real Siegel upper-half plane $\realsiegel$ such that 
 \begin{align}\label{itc:prop_20:form_3}
  \tilde{\permat} = \mfrak(G, \permat) = A\permat \tp{A} + \nicefrac{1}{2}(M-AM\tp{A}).
 \end{align}
 Consider theta-characteristics $\chars{\mathbf{0}}{\bsbeta}$, $\chars{\mathrm{0}}{\tilde{\bsbeta}}$ (of dimension $2n$), whose classes modulo 2 are related as follows:
 \begin{align}\label{itc:prop_20:form_4}
  [\tilde{\bsbeta}]_2 = \afrak \lp G, [\bsbeta]_2\rp = [A\bsbeta + \nicefrac{1}{2}\diag (M-AM\tp{A})]_2.
 \end{align}
 One has then 
 \begin{align}\label{itc:prop_20:form_5}
  \Thefunccar{A\zbf}{\tilde{\permat}}{\mathrm{0}}{\tilde{\bsbeta}} = \Thefunccar{\zbf}{\permat}{\mathrm{0}}{\bsbeta}
 \end{align}
\end{theorem}

\begin{proof}
 Recall that in view of (\ref{form_segno_theta_char}), the value of the theta functions in (\ref{itc:prop_20:form_5}) is invariant as the characteristics vary in a given class modulo 2. One can then assume, without loss of generality,
 \begin{align}
  \tilde{\bsbeta} = A \bsbeta + \nicefrac{1}{2} \diag (M-AM\tp{A}).
 \end{align}
 Plugging this last one and (\ref{itc:prop_20:form_3}) into the definition (\ref{def_theta}) of theta function with characteristics, one obtains
 \begin{align}
  \Thefunccar{A\zbf}{\tilde{\permat}}{\mathrm{0}}{\tilde{\bsbeta}} = & 
  \sum_{\mbf\in \Zbb^n} \exp \lb 
  \pi i \tp{\mbf}\ls A\permat \tp{A}  + \nicefrac{1}{2}\lp M-AM\tp{A}\rp \rs \mbf \right. \\
  &\left. \quad\quad \quad\quad +  2\pi i \tp{\ls A\zbf + \nicefrac{1}{2} A \bsbeta + \nicefrac{1}{4}\diag(M-AM\tp{A}) \rs} \mbf
  \rb \\
  =  & \sum_{\mbf\in \Zbb^n} \exp \lb \pi i \tp{\mbf}(A\permat \tp{A})\mbf + 2\pi i \tp{\ls A(\zbf+\nicefrac{1}{2}\bsbeta)\rs}  \right. \\ 
  &  \left. \quad\quad\quad \quad + \nicefrac{\pi i }{2}  \ls\tp{\mbf}\lp M- AM\tp{A} \rp \mbf + \tp{\diag(M-AM\tp{A})}\mbf  \rs \rb \\
  = & \sum_{\mbf\in \Zbb^n} \exp\lb \pi i \tp{(\tp{A}\mbf)}\permat (\tp{A}\mbf) + 2\pi i \tp{(\zbf+\nicefrac{1}{2}\bsbeta)}(\tp{A}\mbf) \rb \\
  = & \Thefunccar{\zbf}{\permat}{\mathrm{0}}{\bsbeta} 
 \end{align}
  In this last chain of equalities, we have exploited the relation
 \begin{align}
  \tp{\mbf}\lp M-A M \tp{A} \rp \mbf + \tp{\mbf} \diag (M-A M \tp{A}) \equiv 0 \mod 4.
 \end{align}
 In order to deduce this last one, observe that due to (\ref{itc:prop_20:form_5}) 
 \begin{align}
  M - A M \tp{A} = 2 X \quad \quad \mbox{for some} \quad\quad X \in \Mat{n\times n}{\Zbb}, \quad \tp{X}=X.
 \end{align}
 One can then rewrite 
 \begin{align}
  \tp{\mbf}\lp M-A M \tp{A} \rp \mbf + \tp{\mbf} \diag (M-A M \tp{A}) = 2 \ls \tp{\mbf} X \mbf + \tp{\mbf} \diag X \rs = 2 \ls 2 \tp{\mbf} \diag X + 2 c \rs
 \end{align}
 for some appropriate $c$ in $\Zbb$.
 \end{proof}

Next, we restrict to orthosymmetric, real ppavs and prove a purely arithmetic property of the real version (\ref{def_afrak_rid}) of the action $\afrak$.

\begin{theorem}[See Thm. \ref{th_inv_theta_char}] \label{itc:theorem_1}
%\begin{imptheorem} \label{itc:theorem_1}
 Let $(g,\lambda,1)$ be any orthosymmetric type. The sets $\setonow{g}{\lambda}{1}$ and $\setenow{g}{\lambda}{1}$ are invariant for the action $\afrak$ of the group $\realgroupnow{g}{\lambda}{1}$.
%\end{imptheorem}
\end{theorem}

The vectors of the sets $\setonow{g}{\lambda}{1}$ and $\setenow{g}{\lambda}{1}$ can be considered the real (orthosymmetric) analogue of the classical odd and even theta characteristics. 
The proof of theorem \ref{itc:theorem_1} requires some preliminary work. 
\begin{lemma} \label{itc:lemma_1}
 Fix a strictly positive, even integer g equal to $2g_0$. Denote with $M$ the orthosymmetric matrix $\matmnow{2g_0}{2g_0}{1}$. Let $A\in \Mat{2g_0\times 2g_0}{\Zbb}$ satisfy
 \begin{align} \label{itc:lemma_1:form_1}
  AM\tp{A} \equiv M \mod 2.
 \end{align}
 Assume that $\bsbeta, \tilde{\bsbeta} \in \Zbb^{2g_0}$ satisfy
 \begin{align}\label{itc:lemma_1:form_2}
  \tilde{\bsbeta} = A \bsbeta - \frac{1}{2}\diag (AM\tp{A}).
 \end{align}
 Then
 \begin{align}\label{itc:lemma_1:form_3}
  \tp{\tilde{\bsbeta}} M \tilde{\bsbeta} \equiv \tp{\bsbeta} M \bsbeta \mod 4.
 \end{align}
\end{lemma}

\begin{proof}
 According to a theorem of Newman and Smart \cite{newman1964symplectic} there exists a matrix $X$ such that
 \begin{align}\label{itc:lemma_1:form_5}
  X  \in \Sp{2g_0}{\Zbb} \quad\quad \mbox{and} \quad\quad  X \equiv  A \mod 2. 
 \end{align}
 Put
 \begin{align}\label{itc:lemma_1:form_6}
  \hat{\bsbeta} = X\bsbeta -\nicefrac{1}{2} \diag (X M \tp{X})
 \end{align}
 As $\hat{\bsbeta}$ and $\bsbeta$ are congruent modulo 2, one has 
 \begin{align}\label{itc:lemma_1:form_7}
  \tp{\hat{\bsbeta}} M \hat{\bsbeta} \equiv \tp{\tilde{\bsbeta}} M \tilde{\bsbeta} \quad\quad \mod 4.
 \end{align}
 Introduce the notation 
 \begin{align}\label{itc:lemma_1:form_8}
  \bsbeta = \chars{\bsalpha_0}{\bsbeta_0}, \quad \hat{\bsbeta} = \chars{\hat{\bsalpha}_0}{\hat{\bsbeta}_0} \quad \quad \quad \bsalpha_0, \bsbeta_0, \hat{\bsalpha}_0, \hat{\bsbeta}_0 \in \Zbb^{g_0}.
 \end{align}
 Then 
 \begin{align}\label{itc:lemma_1:form_9}
  \tp{\bsbeta} M \bsbeta = 2\tp{\bsalpha_0} \bsbeta_0 
  \quad\quad 
  \tp{\hat{\bsbeta}} M \hat{\bsbeta} = 2\tp{\hat{\bsalpha}_0}\hat{\bsbeta}_0.
 \end{align}
 For the action $\afrak$ defined in (\ref{def_afrak}) one has
 \begin{align}\label{itc:lemma_1:form_10}
  [\hat{\bsbeta}]_2 = \afrak (-JXJ, [\bsbeta]_2) \quad\quad J = \ls \begin{smallmatrix} \mathrm{0} & \mathrm{Id}_{n_0} \\ -\mathrm{Id}_{n_0} & \mathrm{0} \end{smallmatrix} \rs.
 \end{align}
 Indeed, if
 \begin{align}\label{itc:lemma_1:form_11}
  X = \ls \begin{smallmatrix} P & Q \\ R & S \end{smallmatrix} \rs
 \end{align}
 from the condition $X \in \Sp{2g_0}{\Zbb}$, one has
 \begin{align}\label{itc:lemma_1:form_12}
  P\tp{Q} = Q \tp{P}, \quad\quad R\tp{S}= S\tp{R}.
 \end{align}
 which yields
 \begin{align}\label{itc:lemma_1:form_13}
  \hat{\bsbeta} 
  = 
  \ls \begin{smallmatrix} P & Q \\ R & S \end{smallmatrix} \rs \chars{\bsalpha_0}{\bsbeta_0} 
  - 
  \nicefrac{1}{2}\diag \lp 
  \ls\begin{smallmatrix} P & Q \\ R & S \end{smallmatrix}\rs
  \ls\begin{smallmatrix} \mathrm{0} & \mathrm{Id} \\ \mathrm{Id} & \mathrm{0} \end{smallmatrix}\rs
  \ls\begin{smallmatrix} \tp{P} & \tp{R} \\ \tp{Q} & \tp{S} \end{smallmatrix}\rs  
  \rp
  =
  \chars{P\bsalpha_0 + Q \bsbeta_0 - \diag (P\tp{Q})}{R \bsalpha_0 + S \bsbeta_0 - \diag (R\tp{S})}
 \end{align}
 The action $\afrak$ is well-known (\cite{Fay}, \cite{farkas2012theta}) to preserve the parity of the characteristics (\ref{itc:lemma_1:form_8}). That is
 \begin{align}\label{itc:lemma_1:form_14}
  \tp{\hat{\bsalpha}_0} \hat{\bsbeta}_0 \equiv \tp{\bsalpha_0} \bsbeta_0 \mod2.
 \end{align}
 In view of (\ref{itc:lemma_1:form_9}), this last one implies 
 \begin{align}\label{itc:lemma_1:form_15}
  \tp{\hat{\bsbeta}}  M \hat{\bsbeta} \equiv \tp{\bsbeta} M \bsbeta \mod 4.
 \end{align}
 The thesis follows then from (\ref{itc:lemma_1:form_7}) and (\ref{itc:lemma_1:form_15}) together.

\end{proof}

\begin{lemma} \label{itc:lemma_20}
 Fix a type of real ppav $(g,\lambda,\varepsilon)$, with $0<\lambda<g$. Let $A\in \GL{g}{\Zbb}$ satisfy
 \begin{align} \label{itc:lemma_20:form_1}
  A \matm \tp{A} \equiv \matm \mod 2. 
 \end{align}
 Introduce the notation 
 \begin{align}\label{itc:lemma_20:form_2}
  A = \ls
      \begin{smallmatrix}
       A_0 & A_{01} \\ A_{10} & A_1
      \end{smallmatrix}
      \rs.
 \end{align}
 Here $A_0$ and $A_1$ are understood to be square matrices of dimension $\lambda$ and $g-\lambda$ respectively. One has then
 \begin{align}\label{itc:lemma_20:form_3}
  A_0 \matmnow{\lambda}{\lambda}{\varepsilon} \tp{A_0} 
  \equiv
  \matmnow{\lambda}{\lambda}{\varepsilon},
  \quad\quad
  A_{10}\equiv \mathrm{0} 
  \quad\quad
  \mod 2
 \end{align}
 and 
 \begin{align} \label{itc:lemma_20:form_4}
  \diag (A_{10}\matmnow{\lambda}{\lambda}{\varepsilon}\tp{A_{10}}) \equiv 0 \mod 4.
 \end{align}
\end{lemma}

\begin{proof}
 With notation 
 (\ref{itc:lemma_20:form_2}), condition (\ref{itc:lemma_20:form_1}) is equivalent to the following system of equations
 \begin{align}\label{itc:lemma_20:form_5}
  \lb
   \begin{array}{lr}
    A_0\matmnow{\lambda}{\lambda}{\varepsilon}\tp{A_0} \equiv \matmnow{\lambda}{\lambda}{\varepsilon} 
    & \mathrm{mod} 2\\
    A_0 \matmnow{\lambda}{\lambda}{\varepsilon}\tp{A_{10}} \equiv \mathrm{0} 
    & \mathrm{mod} 2\\
    A_{10}\matmnow{\lambda}{\lambda}{\varepsilon} \tp{A_{10}} \equiv \mathrm{0}
    & \mathrm{mod} 2.
   \end{array}
  \right.
 \end{align}
 The first one of (\ref{itc:lemma_20:form_5}) is exactly the first one of (\ref{itc:lemma_20:form_3}) and implies
 \begin{align}
  \det (A_0 \matmnow{\lambda}{\lambda}{\varepsilon}) \equiv 1 \mod 2.
 \end{align}
 The second equation in (\ref{itc:lemma_20:form_5}) implies then the second one of (\ref{itc:lemma_20:form_3}). This last one is equivalent to the existence of a matrix $X\in \Mat{(g-\lambda)\times \lambda}{\Zbb}$ such that
 \begin{align}
  A_{10} = 2X.
 \end{align}
 This yields
 \begin{align}
  \diag(A_{10}\matmnow{\lambda}{\lambda}{\varepsilon}\tp{A_{10}}) = 4\diag (X\matmnow{\lambda}{\lambda}{\varepsilon} \tp{X})
 \end{align}
 which is (\ref{itc:lemma_20:form_4})
\end{proof}

\begin{lemma} \label{itc:lemma_30}
 Fix any orthosymmetric type $(g,\lambda, 1)$. Denote with $M$ the corresponding matrix $\matmnow{g}{\lambda}{1}$. Let $A\in \GL{g}{\Zbb}$ satisfy
 \begin{align} \label{itc:lemma_30:form_1}
  AM\tp{A} \equiv M \mod 2.
 \end{align}
 Consider $\bsbeta\in \Zbb^g$ of the form
 \begin{align} \label{itc:lemma_30:form_2}
  \bsbeta = 
  \ls 
  %\begin{array}{c} \bsbeta_1 \\ \bsbeta_2 \end{array} 
  \begin{smallmatrix} \bsbeta_0 \\ \bsbeta_1 \end{smallmatrix}
  \rs
  , \quad \quad 
  \bsbeta_1 \in \Zbb^{g-\lambda}, \,\,\,\, \bsbeta_1 \equiv \mathbf{0} \mod 2.
 \end{align}
 Put 
 \begin{align} \label{itc:lemma_30:form_3}
  \tilde{\bsbeta} 
  = 
  \ls
  \begin{smallmatrix} 
  \tilde{\bsbeta}_0 \\ \tilde{\bsbeta}_1 
  \end{smallmatrix}
  \rs 
  = A\bsbeta -\nicefrac{1}{2}\diag (A M \tp{A}).
 \end{align}
 Then 
 \begin{align} \label{itc:lemma_30:form_4}
  \tilde{\bsbeta}_1 \equiv \mathbf{0} \mod 2
 \end{align}
 and 
 \begin{align} \label{itc:lemma_30:form_5}
  \tp{\tilde{\bsbeta}_0}\matmnow{\lambda}{\lambda}{1}\tilde{\bsbeta}_0 \equiv \tp{\bsbeta_0}\matmnow{\lambda}{\lambda}{1}\bsbeta_0 \mod 4.
 \end{align}
\end{lemma}
\begin{proof}
 For the ease of notation, let us indicate the matrix $\matmnow{\lambda}{\lambda}{1}$ simply with $M_0$. Making use of notation (\ref{itc:lemma_20:form_2}) and simplifying according to (\ref{itc:lemma_20:form_4}), from the definition (\ref{itc:lemma_30:form_3}) of $\tilde{\bsbeta}$ one obtains
 \begin{align}
  \begin{cases}
   \tilde{\bsbeta}_0 \equiv A_0\bsbeta_0 + A_{01}\bsbeta_1 + \nicefrac{1}{2}\diag (A_0 M_0 \tp A_0) & \mod 2 \\
   \tilde{\bsbeta}_1 \equiv A_1 \bsbeta_1 & \mod 2.
  \end{cases}
 \end{align}
 Relation (\ref{itc:lemma_30:form_4}) follows immediately.
 Now put 
 \begin{align}
  \tilde{\tilde{\bsbeta}}_0 = A_0\bsbeta_0 + \frac{1}{2}\diag(A_0M_0\tp A_0).
 \end{align}
 As 
 \begin{align}
  \tilde{\bsbeta}_0 = \tilde{\tilde{\bsbeta}}_0 \quad\quad \mod 2,
 \end{align}
 one has
 \begin{align} \label{tc_v2:thm_20:form_50}
  \tp{\tilde{\tilde{\bsbeta}}_0} M_0 \tilde{\tilde{\bsbeta}}_0\equiv \tp{\tilde{\bsbeta}_0} M_0 \tilde{\bsbeta}_0 \quad\quad \mod 4
 \end{align}
 On the other side, $A_0, M_0, \bsbeta_0$ and $\tilde{\tilde{\bsbeta}}_0$ satisfy the hypothesis of lemma \ref{itc:lemma_1}. Consequently,
 \begin{align} \label{tc_v2:thm_20:form_60}
  \tp{\tilde{\tilde{\bsbeta}}}_0 M \tilde{\tilde{\bsbeta}}_0 \equiv \tp{\bsbeta_0} M \bsbeta_0 \quad\quad \mod 4.
 \end{align}
 Formulas (\ref{tc_v2:thm_20:form_50}) and (\ref{tc_v2:thm_20:form_60}) together give (\ref{itc:lemma_30:form_5}). The proof is complete.
\end{proof}

\begin{proof}[Proof of theorem \ref{itc:theorem_1}]
 Let 
 %$\ls\begin{array}{c} \mathbf{0} \\ \bsbeta \end{array}\rs_2$ 
 $\ls\begin{smallmatrix} \mathbf{0} \\ \bsbeta \end{smallmatrix}\rs_2$
 be any element of $\setonow{g}{\lambda}{1}$. For the ease of notation, let us put 
 \begin{align}
  \bsbeta = 
  \ls
  \begin{smallmatrix} \bsbeta_1 \\ \bsbeta_2 \end{smallmatrix}
  \rs
  , \quad\quad 
  \bsbeta_1 \in \Zbb^{\lambda}, \bsbeta \in \Zbb^{g-\lambda}.
 \end{align}
 From the definition (\ref{def_set_o}) of $\setonow{g}{\lambda}{1}$ one has 
 \begin{align}
  \bsbeta_2 \equiv \mathbf{0} \mod 2, \quad\quad \tp{\bsbeta_1} \matmnow{\lambda}{\lambda}{1}\bsbeta_1 \equiv 2 \mod 4. 
 \end{align}
 Consider
 \begin{align}
  G = \ls\begin{array}{cc} A & \frac{1}{2}(M-A M \tp{A})(\tp{A})^{-1}\\
                           \mathbf{0} & (\tp{A})^{-1}
  \end{array}\rs
  \in \realgroupnow{g}{\lambda}{1}
 \end{align}
 where, in view of definition (\ref{def_real_mod_group}), 
 \begin{align}
  A \in \GL{g}{\Zbb} 
  \quad\quad \mbox{and} \quad\quad 
  A M \tp{A} \equiv M \mod 2.
 \end{align}
 Put 
 \begin{align}
  \tilde{\bsbeta} = A\bsbeta - \nicefrac{1}{2}\diag (A M \tp{A}).
 \end{align}
 According to definition (\ref{def_afrak_rid}), one has then
 \begin{align}
  %\ls
  %\begin{smallmatrix} \mathbf{0} \\ \bsbeta \end{smallmatrix}
  %\rs_2 
  [\tilde{\bsbeta}]_2
  = \afrak (G, [\bsbeta]_2).
 \end{align}
 This last one is easily seen to belong to $\setonow{g}{\lambda}{1}$. Indeed, let us put 
 \begin{align}
  \tilde{\bsbeta} = 
  \ls
  \begin{smallmatrix} \tilde{\bsbeta}_1 \\ \tilde{\bsbeta}_2 \end{smallmatrix}
  \rs, \quad\quad \tilde{\bsbeta}_1 \in \Zbb^{\lambda}, \tilde{\bsbeta}_2 \in \Zbb^{g-\lambda}.
 \end{align}
 From lemma \ref{itc:lemma_30} one has then 
 \begin{align}
  \tilde{\bsbeta}_2 \equiv \mathbf{0} \mod 2 
 \end{align}
 and
 \begin{align}
  \tp{\tilde{\bsbeta}_1}\matmnow{\lambda}{\lambda}{1}\tilde{\bsbeta}_1 \equiv \tp{\bsbeta_1}\matmnow{\lambda}{\lambda}{1}\bsbeta_1 \mod 4. 
 \end{align}
 The argument for $\setenow{g}{\lambda}{1}$ is analogous.
\end{proof}

We are now ready for the proof of the following

\setcounter{imptheorem}{0}
\begin{theorem}[See Thm. \ref{theorem_1}]\label{itc:theorem_100}
 Fix any orthosymmetric type $(g,\lambda,1)$. Let $\permat, \tilde{\permat} \in \realsiegelnow{g}{\lambda}{1}$ lie in the same orbit of the modular action $\mfrak$ of the real modular group $\realgroupnow{g}{\lambda}{1}$ on $\realsiegelnow{g}{\lambda}{1}$. Then the families $\Ocal (\permat)$, $\Ocal (\tilde{\permat})$ coincide up to a permutation of the indexes. The same holds for the families $\Ecal (\permat)$, $\Ecal (\tilde{\permat})$.
\end{theorem}

\begin{proof}
 Let us first prove the thesis for the families $\Ocal (\permat)$, $\Ocal (\tilde{\permat})$. Let $G$ be an element of $\realgroupnow{g}{\lambda}{1}$ satisfying
 \begin{align} \label{itc:theorem_30:form_1}
  \tilde{\permat} = \mfrak (G, \permat).
 \end{align}
 In view of theorem \ref{itc:theorem_1}, the map
 \begin{align}\label{itc:theorem_30:form_3}
  [\bsbeta]_2 \longrightarrow \afrak(G,[\bsbeta]_2) 
  \quad\quad \quad \quad [\bsbeta]_2\in\Zbb^g_2,
 \end{align}
 reduces to a bijection of $\setonow{g}{\lambda}{1}$ into itself. Moreover, due to theorem \ref{itc:prop_20}, 
 \begin{align}\label{itc:theorem_30:form_4}
  %\Theconst{\tilde{\permat}}{\mathrm{0}}{\tilde{\bsbeta}} = 
  \Theconst{\tilde{\permat}}{\mathrm{0}}{\afrak(G,\bsbeta)} =
  \Theconst{\permat}{\mathrm{0}}{\bsbeta}, 
  \quad\quad\quad\quad
  \forall \bsbeta \in \Zbb^g.
 \end{align}
 One has then 
 \begin{align}\label{itc:theorem_30:form_5}
  \Ocal (\permat) 
  = 
  \lb 
  %\Theta[\bsgamma](\permat)
  \Theconst{\permat}{\mathrm{0}}{\bsbeta}
  \rb_{[\bsbeta]_2\in \setonow{g}{\lambda}{1}} 
  = 
  \lb 
  %\Theta[\tilde{\bsgamma}](\tilde{\permat})
  \Theconst{\tilde{\permat}}{\mathrm{0}}{\afrak(G,\bsbeta)}
  \rb_{[\bsbeta]_2\in \setonow{g}{\lambda}{1}}.
 \end{align}
 On the other side, according to definition (\ref{def_o_tau}),
 \begin{align}\label{itc:theorem_30:form_6}
  \Ocal(\tilde{\permat}) 
  = 
  \lb 
  %\Theta [\bsgamma](\tilde{\permat})
  \Theconst{\tilde{\permat}}{\mathrm{0}}{\bsbeta}
  \rb_{[\bsbeta]_2\in \setonow{g}{\lambda}{1}}.
 \end{align}
 One concludes that the families $\Ocal (\permat)$ and $\Ocal (\tilde{\permat})$ coincide up to the permutation (\ref{itc:theorem_30:form_3})  of their indexes. The proof for families $\Ecal (\permat)$, $\Ecal (\tilde{\permat})$ is analogous.
\end{proof}
 We conclude this section with the following
\begin{lemma} \label{rvt:lemma_1} 
 Fix an orthosymmetric real type $(g,\lambda,1)$. Denote with $M$ the corresponding matrix $\matmnow{g}{\lambda}{1}$. Let $\permat$ be a real 
 Riemann matrix of the real Siegel upper-half space $\realsiegelnow{g}{\lambda}{1}$. The function $\Thf$ (is even and) satisfies
 \begin{align} \label{rvt:lemma_1:form_1}
  \overline{\Thefunc{\zbf}{\permat}} = \Thefunc{\overline{\zbf}}{\permat} \quad\quad\quad\quad \forall \zbf \in \Cbb^g.
 \end{align}
 In particular, $\Thf$ is real on both $\Rsp$ and $\Isp$ introduced in (\ref{intro:def_RI}).
\end{lemma}
\begin{proof}
 Recall that 
 \begin{align}
  \overline{\permat} = M - \permat.
 \end{align}
 This yields 
 \begin{align}
  \overline{\Thefunc{\zbf}{\permat}} 
  &=\overline{\sum_{\mbf\in \Zbb^g}\exp (\pi i \tp{\mbf}\permat\mbf + 2\pi i \tp{\mbf}\zbf})\\
  &=\sum_{\mbf\in\Zbb^g} \exp (-\pi i \tp{\mbf}\overline{\permat}\mbf - 2\pi i \tp{\mbf}\overline{\zbf})\\
  &=\sum_{\mbf\in \Zbb^g} \exp \ls -\pi i \tp{\mbf}(M-\permat)\mbf - 2\pi i \tp{\mbf}\overline{\zbf} \rs\\
  &=\sum_{\mbf\in \Zbb^g} \exp \ls -\pi i \tp{\mbf}M\mbf + \pi i \tp{\mbf}\permat \mbf -2\pi i \tp{\mbf}\overline{\zbf}\rs \\
  &=\sum_{\mbf\in \Zbb^g} \exp(\pi i \tp{\mbf}\permat \mbf - 2\pi i \tp{\mbf}\overline{\zbf})\\
  &= \Thefunc{\overline{\zbf}}{\permat}
 \end{align}
 The parity of the function $\Thf$ is a general, well-known fact. In view of this last one, (\ref{rvt:lemma_1:form_1}) implies
 \begin{align}\label{rvt:lemma_1:form_2}
  \overline{\Thefunc{\zbf}{\permat}} = \Thefunc{-\overline{\zbf}}{\permat}.
 \end{align}
 The reality of $\Thf$ on $\Rsp$ and $\Isp$ follows from (\ref{rvt:lemma_1:form_1}) and (\ref{rvt:lemma_1:form_2}) respectively, after 
 recalling that these last ones are the loci of vectors of $\Cbb^g$ with real and imaginary entries respectively.
\end{proof}

Notice that the elements of the indexed families $\Ocal(\permat)$ and $\Ecal(\permat)$ have the form $\Thefunc{\beta/2}{\permat}$. In view of the previous lemma, these last ones are real, provided that $\permat$ belongs to some real Siegel upper-half space $\realsiegelnow{g}{\lambda}{1}$ space of orthosymmetric type. This concludes the proof of theorem \ref{theorem_1}.

%%%%%%%%%%%%%%%%%%%%% Fourth section  %%%%%%%%%%%%%%%%%%%%%%%%%%%%%%%%%

\section{Sign and zeros of the theta function in the real framework}

 \label{rvt}

 We continue the study of the theta function in the real framework. The results of this section hold for all matrices of some appropriate real Siegel upper-half spaces, without restriction to real period matrices. We start with the following
 \begin{theorem}[See Thm. \ref{intro:thm_last}]
 \label{rvt:thm_1}
  Fix any orthosymmetric real type $(g,\lambda,1)$, with $\lambda>0$. Introduce the set
  \begin{align} \label{rvt:form_1}
   %\setbnow{g}{\lambda}{1} = \lb \bsbeta \mid \bsbeta \in \lb 0,1 \rb^g, \, \tp{\bsbeta}\matmnow{g}{\lambda}{1} \bsbeta = 0, \, \beta_{\lambda+1}= \ldots = \beta_g = 0 \rb.
   %%%%%%%%%%%%%%%%%%
   \setbnow{g}{\lambda}{1} = \lb 
   \bsbeta \,\, | \,\, \bsbeta \in \{0,1\}^g; \,\,
   \beta_k\beta_{\lambda_0+k}=0,\, k=1,2,\ldots\lambda_0;\,\,
   \beta_{\lambda+h}=0,\, h=1,2,\ldots,g-\lambda
   \rb
  \end{align}
  Let $\permat$ be any real Riemann matrix in $\realsiegelnow{g}{\lambda}{1}$. One has
  \begin{align} \label{rvt:form_2}
   \sum_{\bsbeta\in \setbnow{g}{\lambda}{1}} \Thefunc{\nicefrac{1}{2}\bsbeta}{\permat} > 0.  
  \end{align}
  In particular, for every $\permat\in \realsiegelnow{g}{\lambda}{1}$ there exists $\bsbeta \in \setbnow{g}{\lambda}{1}$ such that 
  \begin{align}\label{rvt:form_3}
   \Thefunc{\nicefrac{1}{2}\bsbeta}{\permat} >0.
  \end{align}
 \end{theorem}
 \begin{proof}[Proof of Theorem \ref{intro:thm_last}]
  Let us introduce the notation
  \begin{align}\label{rvt:form_4}
   \syms{g}{\lambda}{1}(\mbf, \bsbeta) = \exp\ls \pi i \lp \nicefrac{1}{2}\tp{\mbf}\matmnow{g}{\lambda}{1}\mbf + \tp{\mbf}\bsbeta \rp \rs, \quad\quad \mbf \in \Zbb^g, \bsbeta \in \setbnow{g}{\lambda}{1}. 
  \end{align}
  One has
  \begin{align} \label{rvt:form_5}
   \sum_{\bsbeta\in \setbnow{g}{\lambda}{1}} \Thefunc{\nicefrac{1}{2}\bsbeta}{\permat} &= \sum_{\bsbeta\in \setbnow{g}{\lambda}{1}} \sum_{\mbf\in \Zbb^g} \exp (\pi i\tp{\mbf}\permat \mbf + \pi i \tp{\mbf}\bsbeta) \\
   & = \sum_{\bsbeta\in \setbnow{g}{\lambda}{1}} \sum_{\mbf\in \Zbb^g} \exp\ls \pi i \lp \nicefrac{1}{2} \tp{\mbf}\matmnow{g}{\lambda}{1} \mbf \rp - \pi \tp{\mbf} T \mbf + \pi i \tp{\mbf}\bsbeta \rs \\
   & = \sum_{\mbf\in \Zbb^g} \ls \exp\lp -\pi \tp{\mbf}T \mbf \rp \sum_{\bsbeta\in \setbnow{g}{\lambda}{1}} \syms{g}{\lambda}{1} (\mbf, \bsbeta) \rs.
  \end{align}
  To prove the thesis, it suffices to show that 
  \begin{align}\label{rvt:form_6}
   \sum_{\bsbeta\in \setbnow{g}{\lambda}{1}} \syms{g}{\lambda}{1}(\mbf, \bsbeta) >0, \quad \quad \quad \forall \mbf\in \Zbb^g.
  \end{align}
  Denote with $\prl$ the projection of an $n$-dimensional vector on its first $\lambda$ components. Notice, that $\prl$ is a bijection from $\setbnow{n}{\lambda}{1}$ onto $\setbnow{\lambda}{\lambda}{1}$. Moreover, 
  \begin{align}\label{rvt:form_7}
   \syms{n}{\lambda}{1}(\mbf, \bsbeta) =  \syms{\lambda}{\lambda}{1} (\prl (\mbf), \prl (\bsbeta))
  \end{align}
  for any $\mbf$ in $\Zbb^g$ and any $\bsbeta$ in $\setbnow{g}{\lambda}{1}$. This implies
  \begin{align}\label{rvt:form_8}
   \sum_{\bsbeta\in \setbnow{g}{\lambda}{1}} \syms{g}{\lambda}{1}(\mathbf{m}, \bsbeta) 
   = \sum_{\bsbeta\in \setbnow{\lambda}{\lambda}{1}} \syms{\lambda}{\lambda}{1} (\prl(\mbf), \prl (\bsbeta)) 
   = \sum_{\bsbeta\in \setbnow{\lambda}{\lambda}{1}} \syms{\lambda}{\lambda}{1}(\prl(\mbf), \bsbeta).
  \end{align}
  So, (\ref{rvt:form_6}) follows, if
  \begin{align} \label{rvt:form_9}
   \sum_{\bsbeta\in \setbnow{\lambda}{\lambda}{1}} \syms{\lambda}{\lambda}{1}(\mbf, \bsbeta) > 0, \quad\quad \quad  \forall \mbf\in \Zbb^{\lambda}.
  \end{align}
  Let us prove this last one. For every $\mbf\in \Zbb^{\lambda}$, let us denote with $\npos{\mbf}$ and $\nneg{\mbf}$ the number of elements $\bsbeta$ of $\setbnow{\lambda}{\lambda}{1}$ such that $\syms{\lambda}{\lambda}{1}(\mbf, \bsbeta)$ equals $1$ or $-1$ respectively:
  \begin{subequations}\label{rvt:form_10}
   \begin{align}
    \npos{\mbf} = \mathrm{card}\lb \bsbeta \mid \bsbeta \in \setbnow{\lambda}{\lambda}{1}; \, \syms{\lambda}{\lambda}{1} (\mbf, \bsbeta)=1 \rb,
   \end{align}
   \begin{align} 
    \nneg{\mbf} = \mathrm{card}\lb \bsbeta \mid \bsbeta \in \setbnow{\lambda}{\lambda}{1}; \, \syms{\lambda}{\lambda}{1} (\mbf, \bsbeta)=-1 \rb.
   \end{align}
  \end{subequations}
  The cardinality of $\setbnow{\lambda}{\lambda}{1}$ is $3^{\lambda_0}$. Indeed, its elements can be described as the vectors $\bsbeta$ of $\{0,1\}^{\lambda}$ such that $(\beta_k, \beta_{\lambda_0+k})$ belongs to $\{(0,0),(0,1),(1,0) \}$, for $k=1,2,\ldots,\lambda_0$. 
  As $\syms{\lambda}{\lambda}{1}(\mbf,\bsbeta)$ can only take values 1 or -1, these facts imply
  \begin{align} \label{rvt:form_11}
   \sum_{\bsbeta\in \setbnow{\lambda}{\lambda}{1}} \syms{\lambda}{\lambda}{1} = \npos{\mbf} - \nneg{\mbf} = 3^{\lambda_0} - 2 \nneg{\mbf}.
  \end{align}
  Let us prove that 
  \begin{align} \label{rvt:form_12}
   \nneg{\mbf} \leq \frac{3^{\lambda_0}-1}{2}, \quad\quad\quad \mbf\in \Zbb^{\lambda}.
  \end{align}
  We proceed by induction on $\lambda_0= \nicefrac{1}{2}\lambda$. For $\lambda_0=1$, one has 
  \begin{align} \label{rvt:form_13}
   \syms{2}{2}{1}(\mbf,\bsbeta) = \exp \ls \pi i \lp m_1m_2 + m_1\beta_1 +m_2\beta_2 \rp \rs 
  \end{align}
  for $\mbf=\tp{(m_1,m_2)}$ in $ \Zbb^2$ and $\bsbeta=\tp{(\beta_1,\beta_2)}$ in 
  $ \setbnow{2}{2}{1}$. It follows, that in this case $\nneg{\mbf}$ is 0 if both $m_1$ and $m_2$ are even and 1 otherwise. This confirms (\ref{rvt:form_9}). Consider now $\lambda_0>1$. Introduce the projections 
  \begin{align} \label{rvt:form_14}
   \prt: \mbf \longrightarrow (m_1,m_{\lambda_0+1})\in \Zbb^2, \quad\quad \quad \prlmt: \mbf \longrightarrow (m_2, \ldots, m_{\lambda_0},m_{\lambda_0+2}, \ldots, m_{\lambda}) \in \Zbb^{\lambda-2}
  \end{align}
  for any $\mbf\in \Zbb^{\lambda}$. One has
  \begin{align}
   \syms{\lambda}{\lambda}{1}(\mbf, \bsbeta) 
   &= \exp \ls \pi i \sum_{k=1}^{\lambda_0}\lp m_k m_{\lambda_0+k} + m_k \beta_k + m_{\lambda_0+k}\beta_{\lambda_0+k} \rp \rs \\
   &= \exp\ls \pi i \lp m_1 m_{\lambda_0+1} + m_1\beta_1 + m_{\lambda_0+1}\beta_{\lambda_0+1} \rp \rs 
   \exp \ls \pi i \sum_{k=2}^{\lambda_0}\lp m_km_{\lambda_0+k} + m_k\beta_k + m_{\lambda_0+k}\beta_{\lambda_0+k} \rp \rs\\
   &= \syms{2}{2}{1}(\prt(\mbf),\prt (\bsbeta))\cdot\syms{\lambda-2}{\lambda-2}{1}(\prlmt(\mbf), \prlmt(\bsbeta)).
  \end{align}
  In view of this decomposition, for any given $\mbf\in \Zbb^{\lambda}$, the vectors $\bsbeta$ of $\setbnow{\lambda}{\lambda}{1}$ such that $\syms{\lambda}{\lambda}{1}(\mbf, \bsbeta)=-1$ are exactly the ones satisfying either of the two following conditions:
  \begin{align}
   \begin{cases}
    \syms{2}{2}{1}(\prt(\mbf), \prt (\bsbeta)) = 1 & \\
    \syms{\lambda-2}{\lambda-2}{1} (\prlmt (\bsbeta), \prlmt (\bsbeta)) = -1 & 
   \end{cases}
   \quad\quad\mathrm{or}\quad\quad
   \begin{cases}
    \syms{2}{2}{1}(\prt(\mbf), \prt (\bsbeta)) = -1 & \\
    \syms{\lambda-2}{\lambda-2}{1} (\prlmt (\bsbeta), \prlmt (\bsbeta)) = 1 & 
   \end{cases}.
  \end{align}
  Moreover, the map
  \begin{align}
   \bsbeta \longrightarrow \lp \prt(\bsbeta), \prlmt (\bsbeta) \rp
  \end{align}
  is a bijection from $\setbnow{\lambda}{\lambda}{1}$ to $\setbnow{2}{2}{1}\times \setbnow{\lambda-2}{\lambda-2}{1}$. As a consequence of these facts and of the inductive hypothesis, for every vector $\mbf$ in $\Zbb^{\lambda}$ one has
  \begin{align}
   \nneg{\mbf} 
   &= \npos{\prt(\mbf)}\nneg{\prlmt(\mbf)} + \nneg{\prt(\mbf)}\npos{\prlmt(\mbf)}\\
   &= \ls 3 -\nneg{\prt(\mbf)} \rs \nneg{\prlmt(\mbf)} + \nneg{\prt(\mbf)}\ls 3^{\lambda_0-1}-\nneg{\prlmt(\mbf)} \rs\\
   &= 3\nneg{\prlmt(\mbf)} + \nneg{\prt(\mbf)} \ls 3^{\lambda_0-1}-2\nneg{\prlmt(\mbf)} \rs\\
   &\leq 3\nneg{\prlmt(\mbf)} + 3^{\lambda_0-1} - 2\nneg{\prlmt(\mbf)}\\
   &= \nneg{\prlmt(\mbf)} + 3^{\lambda_0-1}\\
   & \leq \frac{3^{\lambda_0-1}-1}{2} + 3^{\lambda_0-1} = \frac{3^{\lambda_0}-1}{2}
  \end{align}
  This proves the theorem. 
 \end{proof}

 Notice that $\mathcal{B}_{g,\lambda,1}$ contains exactly one representative for each class of the set $\mathcal{T}_{g,\lambda,1}$ defined in (\ref{def_T_cal}). So, inequalities (\ref{ineq_theta_intro}) and (\ref{rvt:form_2}) are equivalent and theorem \ref{intro:thm_last} follows. We will now exploit the previous result to deduce the existence of points in the spaces $\Isp$ and $\Rsp$ at which the theta function takes a strictly positive value. The following, simple lemma will play a crucial role for the proof of theorem \ref{theorem_2}.  
 \begin{lemma} \label{rvt:lemma_2} 
  Fix an orthosymmetric real type $(g, \lambda,1)$, with $\lambda>0$. Denote with $M$ the corresponding matrix $\matmnow{g}{\lambda}{1}$. Let $\permat$ be a real Riemann 
  matrix in $\realsiegelnow{g}{\lambda}{1}$ and $\bsbeta$ be any vector in $\Zbb^g$. One has\footnote{
  Here we use the convention 
  \begin{align}  \label{rvt:form_2001}
   \sgn x = \begin{cases}
             1 & \text{if } x>0 \\
             0 & \text{if } x=0 \\
             -1 & \text{if } x<0
            \end{cases},
  \quad\quad\quad\quad\quad x\in \Rbb.
  \end{align}
  }
  \begin{align} \label{rvt:form_2002}
   \sgn \Thefunc{\permat(-M\bsbeta)+ \nicefrac{1}{2}\bsbeta}{\permat} 
   = 
   \begin{cases}
    \sgn \Thefunc{\nicefrac{1}{2}\bsbeta}{\permat} & \text{if } \tp{\bsbeta}M\bsbeta \equiv 0 \mod 4\\
    -\sgn \Thefunc{\nicefrac{1}{2}\bsbeta}{\permat} & \text{if } \tp{\bsbeta}M \bsbeta \equiv 2 \mod 4 
   \end{cases}.
  \end{align}
 \end{lemma}
 \begin{proof}
  Recall that
  \begin{align} \label{rvt:form_2003}
   \permat = \nicefrac{1}{2}M + i T
  \end{align}
  for some real, symmetric and positive definite matrix $T$. Applying the quasi-periodicity law (\ref{pertheta}) to the l.h.s. of (\ref{rvt:form_2002}) one obtains
  \begin{align}
    \Thefunc{\permat (-M\bsbeta) + \nicefrac{1}{2}\bsbeta}{\permat} 
    &= \exp\ls -2\pi i \tp{\lp- M\bsbeta\rp}\lp \nicefrac{1}{2}\bsbeta \rp - \pi i \tp{\lp-M\bsbeta\rp}\permat \lp -M\bsbeta \rp \rs \cdot \Thefunc{\nicefrac{1}{2}\bsbeta}{\permat} \\
    &= \exp \ls \pi i \tp{\bsbeta}M\bsbeta - \pi i \tp{\bsbeta}(M\permat M)\bsbeta \rs \cdot \Thefunc{\nicefrac{1}{2}\bsbeta}{\permat} \\
    &= \exp \lp -\nicefrac{1}{2}\pi i \tp{\bsbeta}M \bsbeta \rp \cdot \exp \lp \pi\tp{\bsbeta}M T M \bsbeta \rp \cdot \Thefunc{\nicefrac{1}{2}\bsbeta}{\permat}.
  \end{align}
  The second factor in the last expression is strictly positive. The thesis follows.
 \end{proof}
 
 \begin{corollary}\label{rvt:coro_1} Fix an orthosymmetric real type $(g,\lambda,1)$, with $\lambda>0$. Denote with $M$ the matrix $\matmnow{g}{\lambda}{1}$ defined in (\ref{intro:14}). Let $\permat$ be a real Riemann matrix in $\realsiegelnow{g}{\lambda}{1}$. 
 There exists a point $\xbf_0$ in $\Rsp$ and a point $\zbf_0$ in $\Isp$ such that 
 \begin{align} \label{rvt:inequalities}
  \Thefunc{\xbf_0}{\permat}>0, \quad \quad \quad \Thefunc{\zbf_0}{\permat}>0.
 \end{align}
 Here $\Rsp$ and $\Isp$ are understood to indicate the linear spaces introduced in (\ref{intro:def_RI}).
 \end{corollary}
 \begin{proof}
  According to theorem \ref{rvt:thm_1} there exists $\bsbeta\in \lb0,1\rb^g$ satisfying the conditions
  \begin{align} \label{conditions}
   \tp{\bsbeta} M \bsbeta = 0, \quad\quad \beta_{\lambda+1}=\beta_{\lambda+2}= \ldots = \beta_g =0,
  \end{align}
  such that
  \begin{align}
   \Thefunc{\nicefrac{1}{2}\bsbeta}{\permat}>0.
  \end{align}
  Put $\xbf_0= \nicefrac{1}{2}\bsbeta$. This last one belongs to $\Rsp$ and satisfies the first inequality in (\ref{rvt:inequalities}). Now notice that the second part of conditions (\ref{conditions}) implies 
  \begin{align}
   M(M\bsbeta)=\bsbeta,
  \end{align}
  and put  
  \begin{align}
   \zbf_0 = \permat(-M\bsbeta) + \nicefrac{1}{2}\bsbeta = (\permat-\nicefrac{1}{2}M)(-M\bsbeta).
  \end{align}
  This last one belongs to $\Isp$. Moreover, again in view of conditions (\ref{conditions}) one can apply lemma \ref{rvt:lemma_2} to obtain
  \begin{align}
   \sgn \Thefunc{\zbf_0}{\permat} = \sgn \Thefunc{\permat(-M\bsbeta)+\nicefrac{1}{2}\bsbeta}{\permat} = \sgn \Thefunc{\nicefrac{1}{2}\bsbeta}{\permat} =1.
  \end{align}
  Consequently, $\zbf_0$ satisfies the second inequality in (\ref{rvt:inequalities}).
 \end{proof}

%%%%%%%%%%%%%%%%%%%%%%%%%%%%%%%%%%%%%%%%%%%%%%%%%%%%%%%%%%%%%%%%%%%%%%
%%%%%%%%%%%%%%%%%%%%  End of the proof  %%%%%%%%%%%%%%%%%%%%%%%%%%%%%%

In the remaining part of this section, we will investigate the zeros of the theta function in the real framework. To that purpose, let us recall some well-known facts. Fix any real type $(g,\lambda,\varepsilon)$ and let $\permat$ be a matrix of $\realsiegel$. Consider the real ppav determined by $\permat$ via (\ref{concrete_form}). We will need to consider 
the locus of real and imaginary points on it. That is, of points $[\zbf]_{\Lambda(\permat)}$ in $\Cbb^g/\Lambda(\permat)$ satisfying the equations 
\begin{align}
 S_{conj}([\zbf]_{\Lambda(\permat)}) = [\zbf]_{\Lambda (\permat)} 
 \quad\quad \quad \mbox{or} \quad\quad\quad 
 S_{conj}([\zbf]_{\Lambda(\permat)}) = -[\zbf]_{\Lambda (\permat)} 
\end{align}
respectively. Making use of the sets $\Rsp$ and $\Isp$ defined in (\ref{intro:def_RI}), let us introduce
\begin{align}
 \Rcmp(\permat) = \{ [ \zbf ]_{\Lambda(\permat)} \mid \zbf \in \Rsp \}, 
 \quad\quad\quad
 \Icmp(\permat) = \{ [ \zbf ]_{\Lambda(\permat)} \mid \zbf \in \Isp \}.
\end{align}
In view of definition (\ref{concrete_form_b}), all points of these last ones are real resp. imaginary on the real ppav. If $\lambda<g$, introduce
\begin{align} \label{def_RI_q}
 \Rsp_{\qbf}(\permat) = \lb \permat \bsalpha + \bsbeta \mid \bsalpha, \bsbeta \in \Rbb^g, 2\bsalpha = \left[\begin{smallmatrix} \mathbf{0}\\ \qbf\end{smallmatrix}\right] \rb,
 \quad\quad\quad
 \Isp_{\qbf}(\permat) = \lb \permat \bsalpha + \bsbeta \mid \bsalpha,\bsbeta \in \Rbb^g, M\bsalpha+2\bsbeta = \left[\begin{smallmatrix} \mathbf{0}\\ \qbf\end{smallmatrix}\right] \rb
\end{align}
for $\qbf\in \{ 0,1 \}^{g-\lambda}$ with $\qbf\neq \mathbf{0}$ and $M$ equal to $\matm$. Introduce also the corresponding 
\begin{align} 
 \Rcmp_{\qbf}(\permat) = \{[\zbf]_{\Lambda(\permat)} \mid \zbf\in \Rsp_{\qbf}\}, 
 \quad\quad\quad
 \Icmp_{\qbf}(\permat) = \{[\zbf]_{\Lambda(\permat)} \mid \zbf\in \Isp_{\qbf}\}.
\end{align}
The locus of real (resp. imaginary) points has $2^{g-\lambda}$ connected components given by $\Rcmp(\permat)$ (resp. $\Icmp(\permat)$) 
together with the $(2^{g-\lambda}-1)$ sets $\Rcmp_{\qbf}(\permat)$ (resp. $\Icmp_{\qbf}(\permat)$) if $\lambda<g$ (see \cite{Com155}). \\
Let us recall the  well-known quasi-periodicity property of the $g$-dimensional theta function 
 \begin{align} \label{pertheta}
  \Thefunc{\zbf+\permat\bslambda + \bsmu}{\permat} = \exp (-2\pi i \tp{\bslambda}\zbf - \pi i \tp{\bslambda}\permat \bslambda)\cdot \Thefunc{\zbf}{\permat}  
  \quad \quad \quad \forall \bslambda,\bsmu \in \Zbb^g, \forall \zbf \in \Cbb^g, \forall \permat \in \siegel.
 \end{align}

\begin{theorem}\label{ztf:thm_1}
 Let $(g,\lambda,1)$ be an orthosymmetric real type, with $\lambda<g$. Denote with $\permat$ a matrix of $\realsiegelnow{g}{\lambda}{1}$. Fix a vector $\qbf$ in $\{0,1\}^{g-\lambda}$, $\qbf\neq \mathbf{0}$. The locus 
 \begin{align}
  \lb \zbf \in \Rsp_{\qbf}(\permat) \,\,| \,\, \Thefunc{\zbf}{\permat}=0 \rb
 \end{align}
 is a real analytic variety of dimension $g-1$.
\end{theorem}

\begin{proof}
 In view of definition (\ref{def_RI_q}) it is sufficient to prove the thesis for the locus 
 \begin{subequations}
 \label{ztf:form_1}
 \begin{align}
  \lb \xbf \in \Rbb^g \,\, | \,\, \Thefunc{\xbf+\nicefrac{1}{2}\permat\ebf_{\qbf}}{\permat}=0 \rb
 \end{align}
 where
 \begin{align}
  \ebf_{\qbf} = \left[\begin{smallmatrix} \mathbf{0}\\ \qbf\end{smallmatrix}\right] \in \{0,1\}^g.
 \end{align}
 \end{subequations}
 Introduce the auxiliary function
 \begin{align}\label{ztf:form_2}
  \Tfunc = \exp (\pi i \tp{\ebf_{\qbf}}\cdot\xbf) \Thefunc{\xbf+\nicefrac{1}{2}\permat \ebf_{\qbf}}{\permat} \quad\quad \quad \xbf \in \Rbb^g.
 \end{align}
 First let us show that 
 \begin{align}\label{ztf:form_3}
  \Tfunc\in \Rbb, \quad\quad\quad \forall \xbf \in \Rbb^g.
 \end{align}
 Indeed, denote with $M$ the matrix $\matmnow{g}{\lambda}{1}$ defined in (\ref{intro:14}). As 
 \begin{align}\label{ztf:form_4}
  \overline{\permat} = M-\permat
 \end{align}
 one has
 \begin{align}
  \overline{\Thefunc{\xbf+\nicefrac{1}{2}\permat\ebf_{\qbf}}{\permat}} 
  & = \overline{\sum_{\mbf\in \Zbb^g}\exp\ls \pi i \tp{\mbf}\permat \mbf + 2\pi i \tp{\mbf}(\xbf+\nicefrac{1}{2}\permat\ebf_{\qbf}) \rs} \\
  & = \sum_{\mbf \in \Zbb^g} \exp \ls -\pi i \tp{\mbf}\overline{\permat}\mbf - 2\pi i \tp{\mbf}(\xbf + \nicefrac{1}{2}\overline{\permat}\ebf_{\qbf}) \rs \\
  & = \sum_{\mbf \in \Zbb^g} \exp \lp -\pi i \tp{\mbf}M\mbf + \pi i \tp{\mbf}\permat \mbf - 2\pi i \tp{\mbf}\xbf - \pi i \tp{\mbf}M\ebf_{\qbf} + \pi i \tp{\mbf}\permat \ebf_{\qbf} \rp\\
  & = \sum_{\mbf \in \Zbb^g} \exp \ls \pi i \tp{\mbf}\permat \mbf + 2 \pi i \tp{\mbf}\lp -\xbf + \nicefrac{1}{2}\permat \ebf_{\qbf} \rp \rs \\
  & = \Thefunc{\xbf-\nicefrac{1}{2}\permat\ebf_{\qbf}}{\permat}.
 \end{align}
 On the other side, in view of (\ref{pertheta})
 \begin{align}
  \Thefunc{\xbf-\nicefrac{1}{2}\permat\ebf_{\qbf}}{\permat} 
  & = \Thefunc{(\xbf+\nicefrac{1}{2}\permat \ebf_{\qbf})-\permat \ebf_{\qbf}}{\permat}\\
  & = \exp [ -2\pi i \tp{(-\ebf_{\qbf})}( \xbf +\nicefrac{1}{2}\permat\ebf_{\qbf} ) - \pi i \tp{\ebf_{\qbf}}\permat \ebf_{\qbf} ] \Thefunc{\xbf+\nicefrac{1}{2}\permat\ebf_{\qbf}}{\permat}\\
  & = \exp(2\pi i \tp{\ebf_{\qbf}}\xbf)\Thefunc{\xbf+\nicefrac{1}{2}\permat\ebf_{\qbf}}{\permat}.
 \end{align}
 Combining the two last equations, one obtains 
 \begin{align}
  \overline{\Thefunc{\xbf + \nicefrac{1}{2}\permat\ebf_{\qbf}}{\permat}} = \exp \lp 2\pi i \tp{\ebf_{\qbf}}\xbf \rp \Thefunc{\xbf + \nicefrac{1}{2}\permat\ebf_{\qbf}}{\permat} 
  \quad\quad \quad \xbf\in \Rbb^g.
 \end{align}
 With this last one, one deduces (\ref{ztf:form_3}). Moreover, let $\mathbf{s}$ be a vector in $\{0,1\}^{g-\lambda}$ such that $\tp{\qbf}\cdot \mathbf{s}=1$ and put 
 \begin{align}
  \ebf_{\mathbf{s}} = \left[\begin{smallmatrix} \mathbf{0}\\ \mathbf{s}\end{smallmatrix}\right] \in \{0,1\}^g.
 \end{align}
 Using again (\ref{pertheta}), with a simple manipulation of definition (\ref{ztf:form_2}) one obtains
 \begin{align}\label{ztf:rel_fond}
  \Tfuncnow{\xbf+\ebf_{\mathbf{s}}} = -\Tfunc \quad\quad \quad \xbf\in \Rbb^g.
 \end{align}
 Now notice that $T$ vanishes exactly on the analytic variety defined in (\ref{ztf:form_1}). If this last one had dimension strictly less than $g-1$, its 
 complementary set in $\Rbb^g$ would be path-connected. Consequently, $T$ would be either nonpositive or nonnegative on the whole $\Rbb^g$. But this is not possible: Let $\xbf_0$ be a point of $\Rbb^g$ such that $T(\xbf_0)\neq 0$. In view of (\ref{ztf:rel_fond}), $T(\xbf_0+\ebf_{\mathbf{s}})$ is also different from zero and has opposite sign. As $\permat$ was chosen arbitrarily in $\realsiegelnow{g}{\lambda}{1}$, the proof is complete.
\end{proof}

\begin{lemma} \label{ztf:lemma_1}
 Fix any type $(g,\lambda,\varepsilon)$, with $g\geq 2$. Denote with M the corresponding matrix $\matm$ defined in (\ref{intro:14}). Let $\permat, \tilde{\permat}$ be real Riemann matrices of the real Siegel upper-half space $\realsiegel$. Assume that for some $G\in\realgroup$ one has 
 \begin{align}\label{ztf:lemma_1:form_0}
  \tilde{\permat} = \mfrak(G,\permat).
 \end{align}
 %The function $\Thf$ has zeros on $\Isp_0(\permat)$ if and only if the function $\Thftilde$ has zeros on on $\Isp_0(\tilde{\permat})$.
 The function $\Thf$ has zeros on the imaginary space $\Isp$ defined in (\ref{intro:def_RI}) if and only if $\Thftilde$ does.
\end{lemma}

\begin{proof}
 Recall that $G$ belongs to $\realgroup$ if and only if it has the form 
 \begin{align} \label{ztf:lemma_1:form_1}
  G =   \ls  
  \begin{smallmatrix} 
  A & \nicefrac{1}{2}(M-AM\tp A)(\tp A)^{-1} \\ 
  0 & (\tp A)^{-1} 
  \end{smallmatrix} 
  \rs
 \end{align}
 for some $A\in \GL{g}{\Zbb}$ such that 
 \begin{align} \label{ztf:lemma_1:form_2}
  AM\tp{A} \equiv 0 \mod 2.
 \end{align}
 Correspondingly, condition (\ref{ztf:lemma_1:form_0}) can be recast as follows
 \begin{align}\label{ztf:lemma_1:form_3}
  \tilde{\permat} = A \permat \tp{A} + \nicefrac{1}{2}(M-AM\tp{A}).
 \end{align}
 Let us introduce the map
 \begin{align}\label{ztf:lemma_1:form_4}
  \phi(\zbf) = A \zbf + \nicefrac{1}{4}\diag (M-AM\tp{A}), \quad\quad\quad\quad \zbf\in \Cbb^g.
 \end{align}
 In view of condition (\ref{ztf:lemma_1:form_2}), the constant term on the r.h.s. belongs to $\nicefrac{1}{2}\Zbb^g$. Moreover, in view of lemma \ref{itc:lemma_20}, its last $g-\lambda$ entries are integers. So, both vectors $\pbf$ and $\rbf$ defined as follows
 \begin{align}
  \pbf = 2M\sbf, \quad\quad \rbf = \sbf - \nicefrac{1}{2}M\pbf, \quad\quad        \sbf = \nicefrac{1}{4} \diag (M-AM\tp{A})
 \end{align}
 have only integral entries. With these last ones, one can rewrite
 \begin{align} \label{ztf:lemma_1:form_5}
  \nicefrac{1}{4}\diag (M-AM\tp{A}) 
  = \nicefrac{1}{2}M\pbf + \rbf
  = (\tilde{\permat}-\nicefrac{1}{2}M)(-\pbf) + (\tilde{\permat}\pbf+\rbf).
 \end{align}
 In this way, the l.h.s. has been decomposed as the sum of an imaginary vector and an integer period. Define 
 \begin{align} \label{ztf:lemma_1:form_6}
  g(\zbf) = A\zbf + (\tilde{\permat}-\nicefrac{1}{2}M)(\pbf) \quad\quad \quad \zbf \in \Cbb^g.
 \end{align}
 One has then
 \begin{align} \label{ztf:lemma_1:form_7}
  \phi (\zbf) = g(\zbf) + (\tilde{\permat}\pbf + \rbf) \quad\quad\quad \zbf\in\Cbb^g.
 \end{align}
 Applying theorem \ref{itc:prop_20} and identity (\ref{pertheta}) one obtains 
 \begin{align} \label{ztf:lemma_1:form_8}
  \Thefunc{\zbf}{\permat} &= \Thefunccar{\zbf}{\permat}{\mathrm{0}}{\mathrm{0}} \\
  & = \Thefunccar{A\zbf}{\tilde{\permat}}{\mathrm{0}}{\nicefrac{1}{2}\diag(M-AM\tp{A})} \\
  & = \Thefunc{\phi(\zbf)}{\tilde{\permat}} \\
  & = \Thefunc{g(\zbf)+\tilde{\permat}\pbf+\rbf}{\tilde{\permat}} \\
  & = c(\zbf)\Thefunc{g(\zbf)}{\tilde{\permat}},
 \end{align} 
 for all $\zbf\in \Cbb^g$, $c(\zbf)$ being defined as the never vanishing quantity 
 \begin{align}
  c(\zbf) = \exp [-2\pi i \tp{\mathbf{p}}g(\zbf) - \pi i \tp{\mathbf{p}}\permat \mathbf{p}]
 \end{align}
 It follows that $\Thefunc{\zbf}{\permat}$ vanishes if and only if $\Thefunc{g(\zbf)}{\tilde{\permat}}$ does. 
 Now notice that $g$ reduces to a bijection of the space $\Isp$ into itself. Indeed, the constant term on the r.h.s. of its 
 definition (\ref{ztf:lemma_1:form_6}) belongs to this last one and the matrix $A$ is invertible and real. This concludes the proof. 
\end{proof}

%%%%%%%%%%%%%%%%%%%%%%%%%%%%%%%%%%%%%%%%%%%%%%%%%%%%%%%%%%%%%%%%%%%%%%
%%%%%%%%%%%%%%%%%%%%%% Second section %%%%%%%%%%%%%%%%%%%%%%%%%%%%%%%%

\section{Theta divisor of real Jacobians}

\label{ztf}

This section is dedicated to the study the locus of real and imaginary points of the theta-divisor of real Jacobians (Theorem \ref{intro:th_zeri_RI}). We will first consider separated topological types. 
\begin{proposition}[See Thm. \ref{intro:th_zeri_RI}] \label{ztf:prop_100}
 Let $(\Gamma, \sigma)$ be a real Riemann surface of separated topological type $(g,k,1)$ with $g\geq 2$. Let $\permat$ be a real period matrix of its. One has
 \begin{align}
  \Thefunc{\zbf}{\permat}\neq 0 \quad\quad\quad \forall \zbf \in \Isp.
 \end{align}
\end{proposition}

\begin{proof}
 Recall that, according to (\ref{tttt}), all real period matrices of $(\Gamma,\sigma)$ belong to the real Siegel upper-half space $\realsiegelnow{g}{\lambda}{1}$, where $\lambda$ equals $g+1-k$.
 Let $\Bcal$ be the real, canonical basis in $\homgroup$ with respect to which the (real) period matrix $\permat$ has been computed. 
 If its cycles $a_{\lambda+1},a_{\lambda+2},\ldots, a_n$ are homologically equivalent to ovals of $(\Gamma, \sigma)$, the theorem has already been proved 
 in \cite{dubrovin1988real},\cite{Fay}. Let $\tilde{\permat}$ be any other real period matrix of $(\Gamma,\sigma)$. 
 Then there exists  $G\in \realgroupnow{g}{\lambda}{1}$ such that
 \begin{align}
  \tilde{\permat} = \mfrak(G,\permat).
 \end{align}
 For lemma \ref{ztf:lemma_1}, then, 
 \begin{align}
  \Thefunc{\zbf}{\tilde{\permat}} \neq 0, \quad\quad\quad \forall \zbf \in \Isp.
 \end{align}
 The proof is complete.
\end{proof}
Next, we consider real Riemann surfaces without real points. We will make use of the following, fundamental result.
\begin{theorem}[Riemann, see \cite{Fay}] \label{ztf:real_case:thm_1}
 Let $\Gamma$ be a (compact, complex) Riemann surface and $\mathcal{B}$ a basis of cycles in $\homgroup$.
 There exists a divisor $\riemdiv$ of degree $g-1$ on $\Gamma$ such that for any $P_0\in \Gamma$ and $\ebf\in \Cbb^g$ one has
 \begin{enumerate}
  \item If $\Thefunc{\ebf}{\permat}\neq 0$, there exist exactly $g$ points on $\Gamma$ satisfying the condition 
  \begin{align}
   \Thefunc{\abelmap(P-P_0)-\ebf}{\permat}=0.
  \end{align}
  Denote with $D=P_1+P_2+\ldots+P_g$ the positive divisor of such points. One has 
  \begin{align}
   \ebf \equiv \abelmap (D-P_0-\riemdiv) \quad\quad \quad \mbox{in } \quad \Jac (\Gamma).
  \end{align}
  \item If $\Thefunc{\ebf}{\permat}=0$, then, for some positive divisor $D$ on $\Gamma$ of degree $g-1$, one has 
  \begin{align}
   \ebf \equiv \abelmap (D-\riemdiv) \quad\quad \quad \mbox{in} \quad \Jac(\Gamma).
  \end{align}
 \end{enumerate}
\end{theorem}

In the statement above, the symbol $\mathcal{A}$ denotes the Abel map defined in (\ref{def_Abel_map_b}). We will refer to the divisor $\riemdiv$ as \emph{Riemann divisor}. This last one is in general not positive and depends on the fixed basis of cycles $\mathcal{B}$ in 
$\homgroup$. Recall that the anti-holomorphic involution $\sigma$ on $\Gamma$ easily extends to divisors by linearity: 
\begin{align}
 \sigma \lp \sum_{n=1}^N k_n P_n \rp = \sum_{n=1}^N k_n (\sigma P_n), \quad\quad\quad \sum_{n=1}^N k_n P_n \in \mathrm{Div} (\Gamma).
\end{align}
Let us show that the Riemann divisor behaves well w.r.t. this extension, if $(\Gamma, \sigma)$ is a real Riemann surface without real ovals. 
The following lemma mimics proposition 6.1 of \cite{Fay}.

\begin{lemma} \label{ztf:real_case:lemma_1}
 Let $(\Gamma, \sigma)$ be a real Riemann surface of topological type $(g,0,0)$, with $g\geq 2 $. Fix a real canonical basis of cycles $\Bcal$ in $\homgroup$. Then for the Riemann divisor one has 
 \begin{align}
  \riemdiv - \sigma \riemdiv \equiv 0 \quad\quad \mbox{in } \Pic (\Gamma)
 \end{align}
\end{lemma}

Recall that $J_0(\Gamma)$ was defined in (\ref{def_Pic}) as the quotient of the group of divisors of degree zero w.r.t. the subgroup of the principal ones.

\begin{proof}
 Denote with $\permat$ the (real) period matrix corresponding to $\Bcal$. In view of (\ref{tttt}), this last one belongs to $\realsiegelnow{g}{g}{1}$ or $\realsiegelnow{g}{g-1}{1}$, depending on $g$ being even or odd respectively. So, from lemma \ref{rvt:lemma_1}, one has
 \begin{align} \label{ztf:real_case:lemma_1:form_1}
  \Thefunc{\overline{\zbf}}{\permat} 
  = \overline{\Thefunc{\zbf}{\permat}} \quad\quad\quad \forall \zbf\in \Cbb^g.
 \end{align}
 Fix $P_0\in \Gamma$ and $\ebf\in \Cbb^g$ such that $\Thefunc{\ebf}{\permat}\neq 0$. Let
 \begin{align}\label{ztf:real_case:lemma_1:form_2}
  D = P_1+P_2+\ldots +P_g
 \end{align}
 be the positive divisor of the g points on $\Gamma$ satisfying the condition
 \begin{align} \label{ztf:real_case:lemma_1:form_3}
  \Thefunc{\abelmap(P-P_0)-\ebf}{\permat}=0.
 \end{align}
 In view of symmetry (\ref{ztf:real_case:lemma_1:form_1}), the points of the divisor 
 \begin{align} \label{ztf:real_case:lemma_1:form_4}
  \sigma D = \sigma P_1+\sigma P_2 + \ldots +\sigma P_g
 \end{align}
 are exactly the ones satisfying
 \begin{align} \label{ztf:real_case:lemma_1:form_5}
  \Thefunc{\abelmap(P-\sigma P_0) - \overline{\ebf}}{\permat} = 0.
 \end{align}
 According to theorem \ref{ztf:real_case:thm_1}, conditions (\ref{ztf:real_case:lemma_1:form_3}) and (\ref{ztf:real_case:lemma_1:form_5}) imply 
 \begin{align}
  \ebf \equiv \abelmap (D-P_0-\riemdiv) 
  \quad\quad \mbox{and} \quad\quad
  \overline{\ebf}\equiv \abelmap(\sigma D - \sigma P_0 - \riemdiv)
  \quad\quad\quad \mbox{in} \,\,\, \Jac (\Gamma). 
 \end{align}
 The thesis follows from an elementary manipulation of these last ones.
\end{proof}
Let $(\Gamma,\sigma)$ be a real Riemann surface of topological type $(g,0,0)$, with $g\geq 2$. Fix a real, canonical basis of cycles $\Bcal$ in $\homgroup$. Indicate with $\permat$ the corresponding real period matrix of $\Gamma$. The results above allow a convenient characterization of\footnote{This set is of course well-defined, in view of (\ref{pertheta}).}
\begin{align}
 \thetadivreal = \lb [\zbf]_{\Lambda (\permat)} \in \mathrm{Jac}(\Gamma) \mbox{ such that }
 S_{conj} ([\zbf]_{\Lambda (\permat)}) = [\zbf]_{\Lambda(\permat)} \mbox{ and } \Thefunc{\zbf}{\permat}=0
 \rb
\end{align}
Denote with $\Gamma^{g-1}$ the set of positive divisors on $\Gamma$ of degree $g-1$. 
Introduce
\begin{align}
 \Zcal (\Gamma) := \lb D \in \Gamma^{g-1} \mid D-\sigma D \equiv 0 \,\,\mbox{in} \,\,  \Pic(\Gamma) \rb.  
\end{align}
As a consequence\footnote{See \cite{farkas1992riemann} and \cite{siegel_vol1} for more detail.} of theorem \ref{ztf:real_case:thm_1}, one has 
\begin{align}
 \thetadivreal = \lb \abelmap(D-\riemdiv) \mid D \in \Zcal (\Gamma) \rb
\end{align}
Now notice that this last one can be decomposed as 
\begin{align}
 \Zcal(\Gamma) = \Zcal_{triv}(\Gamma) \cup \Zcal_{nontriv}(\Gamma)
\end{align}
where
\begin{align}
 \Zcal_{triv} = \lb D \in \Zcal (\Gamma) \mid D = \sigma D \,\,\mbox{as plane divisors} \rb \\
 \Zcal_{nontiv} = \lb D \in \Zcal (\Gamma) \mid D \neq \sigma D \,\, \mbox{as plane divisors} \rb
\end{align}
We will consider the points of $\thetadivreal$ yielded by each of these two sets.
\begin{proposition} \label{ztf:real_case:prop_1}
 The set
 \begin{align} \label{ztf:real_case:set_2}
  \lb \abelmap (D-\riemdiv) \mid D \in \Zcal_{nontriv}(\Gamma)\rb
 \end{align}
 is contained in a real analytic subvariety of the locus of 
 real points of $(\Jac (\Gamma), S_{conj})$ whose real codimension is at least three.
\end{proposition}
\begin{proof}
 The conditions defining $\Zcal_{triv} (\Gamma)$ imply that for each divisor $D$ of its there exists a nonconstant element in $L(D)$, the space of meromorphic functions on $\Gamma$ with poles bounded by $-D$. That is
 \begin{align}
  l (D)>1.
 \end{align}
 Introduce 
 \begin{align}
  \mathcal{G}(\Gamma) := \lb D\in \Gamma^{g-1} \mid l(D)>1 \rb.
 \end{align}
 Denote with $\mathcal{W}(\Gamma)$ the set of zeros of $\Thf$ it generates:
 \begin{align}
  \mathcal{W}(\Gamma) := \lb \abelmap (D-\riemdiv) \mid D \in \mathcal{G}(\Gamma) \rb.
 \end{align}
 Now notice that
 \begin{align}
  \sigma (\mathcal{G}(\Gamma)) \subset \mathcal{G}(\Gamma).
 \end{align}
 Lemma \ref{ztf:real_case:lemma_1} and the definition of the symmetry $S_{\Gamma}$  on $\Jac (\Gamma)$ (see section \ref{intro}) imply then
 \begin{align}
  S_{\Gamma} (\mathcal{W}) \subset \mathcal{W}.
 \end{align}
 That is, $\mathcal{W}$ inherits an induced anti-holomorphic involution from the real Jacobian of $(\Gamma, \sigma)$. On the other hand, $\mathcal{W}$ is an 
 analytic subvariety of $\Jac (\Gamma)$ of dimension at most $g-3$. Indeed, it is a result of Martens \cite{Martens1967}, that a translate of its is. The thesis 
 follows then by observing that the elements of the set (\ref{ztf:real_case:set_2}) are $S_{\Gamma}$-real points of the subvariety $\mathcal{W}$ and by recalling that 
 $S_{\Gamma}$ coincides with $S_{conj}$, because the basis of cycles $\Bcal$ in $\homgroup$ is real.  
\end{proof}
\begin{proposition} \label{ztf:real_case:prop_2}
 The set $\Zcal_{triv}(\Gamma)$ is nonempty if and only if $g$ is odd. In that case, its image 
 \begin{align} \label{ztf:real_case:prop_2:quad_01}
  \lb \abelmap (D-\riemdiv) \mid  D \in \Zcal_{triv}(\Gamma) \rb
 \end{align}
 is entirely contained in $\Rcmp_1 (\permat)$
\end{proposition}
\begin{proof}
 Let $D$ be a positive divisor on $\Gamma$ such that $\sigma D = D$ (as plane divisors). If $P$ is a point of $D$, then also $\sigma P$ is. $P$ and $\sigma P$ are 
 distinct, because the symmetry $\sigma$ has no fixed points. One has then 
 \begin{align}
  D = (P+\sigma P ) +\tilde{D}
 \end{align}
 for some (possibly vanishing), positive divisor $\tilde{D}$ such that $\sigma \tilde{D}=\tilde{D}$ and $\deg \tilde{D} = \tilde{D}$. By induction, one can then 
 easily prove that $\deg D $ is even. This forces the genus of $\Gamma$ to be odd.\\
 As $\Zcal_{triv}(\Gamma)$ is connected, the set (\ref{ztf:real_case:prop_2:quad_01})
 is also connected. Consequently, it entirely lies either in $\Rcmp_1(\permat)$ or in $\Rcmp(\permat)$. If this last one was the case, the points of $\thetadivreal$ 
 in $\Rcmp_1(\permat)$ would all correspond to divisors of $\Zcal_{nontriv}(\Gamma)$. In view of proposition \ref{ztf:real_case:prop_1}, 
 $\thetadivreal \cap \Rcmp_1(\permat)$ would be contained in a real analytic subvariety of dimension not greater than $g-3$. But this would contradict theorem 
 \ref{ztf:thm_1}, because $\Rsp_1(\permat)$ is locally isomorphic to $\Rcmp_1(\permat)$. In this way, also the second part of the thesis is proved.
\end{proof}
As an easy consequence of the last two propositions one has the following 
\begin{theorem} \label{ztf:real_case:thm_10}
 Let $(\Gamma,\sigma)$ be a real Riemann surface of topological type $(g,0,0)$, with $g\geq 2$. Denote with $\permat$ a real period matrix of its. The set
 \begin{align} \label{ztf:real_case:thm_10_form_1}
  \lb \zbf \in \Rsp \mid \Thefunc{\zbf}{\permat} \neq 0 \rb
 \end{align}
 is path-connected.
\end{theorem}
\begin{proof}
 In view of proposition \ref{ztf:real_case:prop_1}, all the points of $\thetadivreal \cap \Rcmp(\permat)$ correspond to divisors in $\Zcal_{nontriv}(\Gamma)$. 
 In view of proposition \ref{ztf:real_case:prop_2}, these points are contained in a real subvariety of (real) dimension not greater than $g-3$. As $\Rcmp(\permat)$ 
 is a manifold of real dimension $g$, the set $\Rcmp(\permat)\backslash \thetadivreal$ is path-connected. The thesis follows then by observing that the set 
 in (\ref{ztf:real_case:thm_10_form_1}) is locally isomorphic to this last one.
\end{proof}

%%%%%%%%%%%%%%%%%%%%%%%%%%%%%%%%%%%%%%%%%%%%%%%%%%%%%%%%%%%%%%%%%%%%%%%

 \section{Proof of the main theorem}
 
 \label{fin}
 
 We are now ready to prove Theorem \ref{theorem_2}. Notice that, according to 
 definitions (\ref{def_theta}) and (\ref{def_theta_const}), for any $g\geq 2$ one has 
 \begin{align}
  \label{fin:form_1}
  \Theconst{\permat}{\mathrm{0}}{\bsbeta} = \Thefunc{\nicefrac{\bsbeta}{2}}{\permat}
 \end{align}
 for all $\bsbeta \in \Zbb^g$ and $\permat \in \siegel$.
\begin{proposition} \label{fin:prop_1}
 Let $(\Gamma, \sigma)$ be a real Riemann surface of topological type $(2g_0,1,1)$ or $(2g_0+1,2,1)$, for some integer $g_0\geq 1$. Let $\permat$ be a real period matrix of its. One has 
 \begin{align}
  \Thefunc{\zbf}{\permat}>0, \quad\quad\quad \forall \zbf \in \Isp.
 \end{align}
\end{proposition}

\begin{proof}
 According to (\ref{tttt}), $\permat$ belongs to $\realsiegelnow{2g_0}{2g_0}{1}$ or $\realsiegelnow{2g_0+1}{2g_0}{1}$ respectively. In view of 
 Proposition \ref{ztf:prop_100} the function $\Thefunc{\zbf}{\permat}$ is nonzero for every $\zbf$ in $\Isp$. Moreover, it is continuous and, due to 
 lemma \ref{rvt:lemma_1}, real on $\Isp$. So, it is either strictly positive or strictly negative on the whole $\Isp$. Corollary \ref{rvt:coro_1} 
 implies that the former option holds here.
\end{proof}
\begin{proposition}\label{fin:prop_2} Let $(\Gamma,\sigma)$ be a real Riemann surface of topological type $(g,0,0)$, for some integer $g\geq 2$. Let $\permat$ be a period matrix of its. One has 
 \begin{align}
  \Thefunc{\zbf}{\permat}\geq 0, \quad\quad\quad \forall \zbf \in \Rsp.
 \end{align}
\end{proposition}
\begin{proof}
 Again, according to (\ref{tttt}), the matrix $\permat$ belongs to $\realsiegelnow{2g_0}{2g_0}{1}$ or $\realsiegelnow{2g_0+1}{2g_0}{1}$, with $g$ equal 
 to $2g_0$ or $2g_0+1$ for some integer $g_0\geq 1$ respectively. In view of theorem \ref{ztf:real_case:thm_10}, the real analytic variety
 \begin{align}
  \lb \zbf \mid \zbf \in \Rsp, \Thefunc{\zbf}{\permat}=0 \rb
 \end{align}
 has real dimension not greater than $g-3$. Moreover, the theta function in the last expression is continuous and, due to lemma \ref{rvt:lemma_1}, real on $\Rsp$.
 As a consequence, this last one is either nonnegative or nonpositive there. Corollary \ref{rvt:coro_1} implies that the former case holds here.
\end{proof}
\begin{proposition}\label{fin:prop_3}
 Let $(\Gamma, \sigma)$ be a real Riemann surface of topological type $(2g_0,1,1)$ or $(2g_0+1,2,1)$, for some integer $g_0\geq 1$. Let $\permat$ be a real period 
 matrix of its. One has
 \begin{subequations}\label{inequalities}
  \setcounter{equation}{14}
  \begin{align}\label{ineq_odd}
   \Theconst{\permat}{\mathrm{0}}{\bsbeta}<0, \quad\quad \quad \mbox{if } \,\, [\bsbeta]_2 \in \setonow{g}{2g_0}{1}
  \end{align}
  and
  \setcounter{equation}{4}
  \begin{align} \label{ineq_even}
   \Theconst{\permat}{\mathrm{0}}{\bsbeta}>0, \quad\quad \quad \mbox{if } \,\, [\bsbeta]_2 \in \setenow{g}{2g_0}{1}
  \end{align}
 \end{subequations}
\end{proposition}
Recall that the sets $\setonow{g}{2g_0}{1}$ and $\setenow{g}{2g_0}{1}$ have been defined in (\ref{def_set_oe}) and that in view of (\ref{form_segno_theta_char}) 
the theta constants in (\ref{inequalities}) are independent of the particular representative of $[\bsbeta]_2$.

\begin{proof}
 First consider the case of a real Riemann surface of topological type $(2g_0+1,2,1)$. According to (\ref{tttt}) all real period matrices of this last 
 one are real Riemann matrices of the real Siegel upper-half plane $\realsiegelnow{2g_0+1}{2g_0}{1}$. Denote with $M$ the matrix $\matmnow{2g_0+1}{2g_0}{1}$ defined in 
 (\ref{intro:14}). In view of (\ref{fin:form_1}), for (\ref{ineq_odd}) it suffices to prove that 
 \begin{align}
  \Thefunc{\nicefrac{\bsbeta}{2}}{\permat}<0
 \end{align}
 for all $\bsbeta\in  \Zbb^{2g_0+1}$ such that $\tp{\bsbeta}M\bsbeta\equiv 2 \mod 4$ and $\beta_{2g_0+1}=0$. In particular, this last condition implies 
 \begin{align}
  M(M\bsbeta) = \bsbeta.
 \end{align}
 Due to lemma \ref{rvt:lemma_2}, one can equivalently show for such vectors $\bsbeta$ the following inequality
 \begin{align}
  0<\Thefunc{\permat(-M\bsbeta)+\nicefrac{1}{2}\bsbeta}{\permat} = \Thefunc{(\permat-\nicefrac{1}{2}M)(-M\bsbeta)}{\permat}.
 \end{align}
 Notice that the argument of the theta function on the r.h.s. belongs to the space $\Isp$. This inequality follows then from proposition \ref{fin:prop_1}. The case of 
 real Riemann surfaces of topological type $(2g_0,1,1)$ as well as inequality (\ref{ineq_even}) can be treated similarly.
 \end{proof}
 \begin{theorem}[See Thm. \ref{theorem_2}]
  \label{fin:thm_ultimo}
  Let $(\Gamma, \sigma)$ be a real Riemann surface of genus greater or equal than two and $\permat$ a real period matrix of its.  
  If the topological type of $(\Gamma, \sigma)$ is $(2g_0, 1, 1)$ or $(2g_0+1,2,1)$, then all elements of the indexed family $\Ocal (\permat)$ are strictly negative. 
  If instead the topological type of $(\Gamma,\sigma)$ is $(g,0,0)$, then all elements of $\Ocal (\permat)$ are positive, some of them possibly vanishing.
 \end{theorem}

 \begin{proof} 
  The part concerning real Riemann surfaces of separated topological type is already contained in proposition \ref{fin:prop_3}. If $\permat$ is the real period matrix of a real Riemann surface without real points, one has
  \begin{align}
   \Thefunc{\nicefrac{\bsbeta}{2}}{\permat} \geq 0, \quad\quad\quad \forall \bsbeta\in \Zbb^g.
  \end{align}
  This follows from proposition \ref{fin:prop_2}, after observing that all points of the form $\nicefrac{1}{2}\bsbeta$ for $\bsbeta$ in $\Zbb^g$ belong 
  to $\Rsp$. The thesis follows then from identity (\ref{fin:form_1}).
 \end{proof}
 With an analogous argument, one proves the following
 \begin{theorem}\label{thm_E_quantities}
  Let $(\Gamma,\sigma)$ be a real Riemann surface of genus greater or equal than two and $\permat$ a real period matrix of its. If the topological type of 
  $(\Gamma,\sigma)$ is $(2g_0,1,1)$ or $(2g_0+1,2,1)$, then all elements of the indexed family $\Ecal (\permat)$ are strictly positive. If instead the topological type is 
  $(g,0,0)$, then all elements of $\Ecal(\permat)$ are nonnegative.
 \end{theorem}

\section{Acknowledgment}

The author wishes to thank Prof. V. Kharlamov and Prof. H. Braden for their valuable feedback about the research plan, Prof. B. Dubrovin and Prof. K. Hulek for supportive discussions during the investigations, and Prof. A. Moro and Prof. M. Mazzocco for their precious help by the revision and improvement of the manuscript.

%%%%%%%%%%%%%%%%%%%%%%%%%%%%%%%%%%%%%%%%%%%%%%%%%%%%%%%%%%%%%%%%%%%%%%%
%%%%%%%%%%%%%%%%%%%%%%   Bibliography  %%%%%%%%%%%%%%%%%%%%%%%%%%%%%%%%

\bibliographystyle{plain} 
\bibliography{./math_ref}

\end{document}